\documentclass[10pt,a4paper]{amsart}
\usepackage[utf8]{inputenc}
\usepackage[T1]{fontenc}
\usepackage{lmodern}
\usepackage{mathrsfs}
\usepackage{amsmath,amssymb,amsfonts}
\usepackage[x11names]{xcolor}
\usepackage{enumitem}
\usepackage{multicol}
\usepackage{euscript}

\usepackage[unicode,pdfborder={0 0 0},final]{hyperref}
\hypersetup{
    colorlinks = true,
    linkcolor = DarkOrchid3,
    urlcolor  = Coral2,
    citecolor = RoyalBlue3
}

\usepackage{amsthm}
\usepackage{thmtools}
\makeatletter
\renewcommand\thmt@autorefsetup{%
  \@xa\def\csname\thmt@envname autorefname\@xa\endcsname\@xa{\thmt@thmname}%
}
\makeatother
\declaretheoremstyle[
	spaceabove=6pt plus 3pt minus 3pt, spacebelow=6pt plus 3pt minus 3pt,%
	headfont=\normalfont\bfseries,%
	notefont=\normalfont, notebraces={(}{)},%
	bodyfont=\normalfont,%
	postheadspace=.5em,
]{defn}
\declaretheorem[numberwithin=section,name=Definition,style=defn]{defi}

\declaretheorem[sibling=defi,name=Example,style=defn]{exe}

\declaretheoremstyle[
	spaceabove=6pt plus 3pt minus 3pt, spacebelow=6pt plus 3pt minus 3pt,
	headfont=\normalfont\bfseries,
	notefont=\normalfont, notebraces={(}{)},
	bodyfont=\itshape,
	postheadspace=.5em,
	qed=
]{proposition}
\declaretheorem[sibling=defi,name=Proposition,style=proposition]{prop}

\declaretheorem[sibling=defi,name=Theorem,style=proposition]{theo}
\declaretheorem[numbered=no,name=Theorem,style=proposition]{theonn}
\declaretheorem[sibling=defi,name=Lemma,style=proposition]{lem}
\declaretheorem[sibling=defi,name=Corollary,style=proposition]{cor}

\declaretheoremstyle[
	spaceabove=6pt plus 3pt minus 3pt, spacebelow=6pt plus 3pt minus 3pt,
	headfont=\normalfont\itshape,
	notefont=\normalfont, notebraces={(}{)},
	bodyfont=\normalfont,
	postheadspace=.5em,
	qed=
]{remarque}
\declaretheorem[sibling=defi,name=Remark,style=remarque]{rema}

\makeatletter
\DeclareFontFamily{OMX}{MnSymbolE}{}
\DeclareSymbolFont{MnLargeSymbols}{OMX}{MnSymbolE}{m}{n}
\SetSymbolFont{MnLargeSymbols}{bold}{OMX}{MnSymbolE}{b}{n}
\DeclareFontShape{OMX}{MnSymbolE}{m}{n}{
    <-6>  MnSymbolE5
   <6-7>  MnSymbolE6
   <7-8>  MnSymbolE7
   <8-9>  MnSymbolE8
   <9-10> MnSymbolE9
  <10-12> MnSymbolE10
  <12->   MnSymbolE12
}{}
\DeclareFontShape{OMX}{MnSymbolE}{b}{n}{
    <-6>  MnSymbolE-Bold5
   <6-7>  MnSymbolE-Bold6
   <7-8>  MnSymbolE-Bold7
   <8-9>  MnSymbolE-Bold8
   <9-10> MnSymbolE-Bold9
  <10-12> MnSymbolE-Bold10
  <12->   MnSymbolE-Bold12
}{}

\let\llangle\@undefined
\let\rrangle\@undefined
\DeclareMathDelimiter{\llangle}{\mathopen}%
                     {MnLargeSymbols}{'164}{MnLargeSymbols}{'164}
\DeclareMathDelimiter{\rrangle}{\mathclose}%
                     {MnLargeSymbols}{'171}{MnLargeSymbols}{'171}
\makeatother

\usepackage{tikz}
\usetikzlibrary{arrows,matrix,cd,babel}
\usepackage{quiver}

\newcommand{\Zb}{\mathbb{Z}}
\newcommand{\Rb}{\mathbb{R}}

\newcommand{\Cb}{\mathbb{C}}

\newcommand{\Pb}{\mathbb{P}}
\newcommand{\Abb}{\mathbb{A}}

\newcommand{\Er}{\mathrm{E}}

\newcommand{\Kr}{\mathrm{K}}

\newcommand{\GW}{\mathrm{GW}}

\newcommand{\Sr}{\mathrm{S}}

\newcommand{\CH}{\mathrm{CH}}
\newcommand{\Hr}{\mathrm{H}}

\newcommand{\Cr}{\mathrm{C}}

\newcommand{\Kbf}{\mathbf{K}}
\newcommand{\Ibf}{\mathbf{I}}
\newcommand{\Ir}{\mathrm{I}}

\newcommand{\Fr}{\mathrm{F}}

\newcommand{\Wr}{\mathrm{W}}
\newcommand{\et}{\textup{\'et}}

\newcommand{\Wbf}{\mathbf{W}}

\newcommand{\Ch}{\mathrm{Ch}}

\newcommand{\alg}{\mathrm{alg}}
\newcommand{\Osc}{\mathscr{O}}

\newcommand{\Sm}{\mathrm{Sm}}

\newcommand{\Fbf}{\mathbf{F}}

\newcommand{\Hsc}{\mathscr{H}}

\newcommand{\Fsc}{\mathscr{F}}

\newcommand{\Zar}{\mathrm{Zar}}

\newcommand{\Vsc}{\mathscr{V}}

\newcommand{\Leu}{\EuScript{L}}

\DeclareMathOperator{\Ker}{Ker}

\DeclareMathOperator{\Id}{Id}

\DeclareMathOperator{\Spec}{Spec}

\DeclareMathOperator{\Coker}{Coker}

\DeclareMathOperator{\sign}{sign}

\DeclareMathOperator{\Pic}{Pic}

\DeclareMathOperator{\Sq}{Sq}

\DeclareMathOperator{\Bra}{Br}

\renewcommand{\Im}{\operatorname{Im}}

\usepackage{appendix}

\usepackage[backend=biber,style=ext-alphabetic, maxbibnames=50,useprefix=true]{biblatex}
\DefineBibliographyExtras{english}{\DeclarePunctuationPairs{comma}{*?!}}

\DeclareBibliographyDriver{article}{%
  \usebibmacro{bibindex}%
  \usebibmacro{begentry}%
  \usebibmacro{author}%
  \setunit{\addcomma\space}
  \printfield[emph]{title}
  \setunit{\addcomma\space}
  \printfield{journaltitle}
  \setunit{\addspace}%
  \textbf{\printfield{volume}}
  \setunit{\addspace}%
  \printfield[parens]{year}
  \setunit{\addcomma\space}
  \printfield{number}
  \setunit{\addcomma\space}
  \printfield{pages}
  \usebibmacro{finentry}%
}

\DeclareFieldFormat[article]{title}{#1}
\DeclareFieldFormat[article]{journaltitle}{#1}
\DeclareFieldFormat[article]{pages}{#1}
\DeclareFieldFormat[article]{number}{\bibstring{number}~#1}

\DeclareBibliographyDriver{book}{%
  \usebibmacro{bibindex}%
  \usebibmacro{begentry}%
  \printnames{author}
  \setunit{\addcomma\space}
  \printfield{title}
  \setunit{\addcomma\space}
  \printfield{edition}
  \setunit{\addcomma\space}
  \printfield{series}
  \setunit{\addcomma\space}
  \printfield{volume}
  \setunit{\addcomma\space}
  \printfield{number}
  \setunit{\addcomma\space}
  \printlist{publisher}
  \setunit{\addcomma\space}
  \printlist{location}
  \setunit{\addcomma\space}
  \printfield{year}
  \usebibmacro{finentry}%
}

\DeclareFieldFormat[book]{volume}{\bibstring{volume}~#1}
\DeclareFieldFormat[book]{number}{\bibstring{number}~#1}
\DeclareFieldFormat[book]{publisher}{#1}

\DeclareBibliographyDriver{incollection}{%
  \usebibmacro{bibindex}%
  \usebibmacro{begentry}%
  \printnames{author}
  \setunit{\addcomma\space}
  \printfield{title}
  \setunit{\addperiod\space}
  \usebibmacro{in:}
  \printfield[emph]{booktitle}
  \setunit{\addcomma\space}
  \usebibmacro{byeditor+others}
  \setunit{\addcomma\space}
  \printfield{edition}
  \setunit{\addcomma\space}
  \printfield{series}
  \setunit{\addcomma\space}
  \printfield{volume}
  \setunit{\addcomma\space}
  \printfield{number}
  \setunit{\addcomma\space}
  \printlist{publisher}
  \setunit{\addcomma\space}
  \printlist{location}
  \setunit{\addcomma\space}
  \printfield{year}
  \setunit{\addcomma\space}
  \printfield{pages}
  \usebibmacro{finentry}%
}

\DeclareFieldFormat[incollection]{volume}{\bibstring{volume}~#1}
\DeclareFieldFormat[incollection]{number}{\bibstring{number}~#1}
\DeclareFieldFormat[incollection]{series}{#1,\addspace}

\DefineBibliographyExtras{english}{%
}%

\numberwithin{equation}{section}

\renewcommand*\bar[1]{\rlap{$\smash{\overline{#1}}$}\phantom{#1}}

\usepackage[foot]{amsaddr}
\title{Witt groups of smooth real curves and surfaces}
\author{Samuel Lerbet}
\address{DMA, École normale supérieure, Université PSL, CNRS, 75005 Paris, France}
 \email{samuel.lerbet@ens.fr}
\date{\today}

\begin{document}

\begin{abstract}
We study the $\Ibf^*$-cohomology of a smooth real algebraic curve in terms of its real locus and its geometric genus. We notably extend results of Knebusch and Monnier to the twisted case. We further describe how to compute twisted Witt groups of surfaces, generalising work of Sujatha, and apply our results to the computation of the shifted and twisted Witt groups of the smooth real anisotropic quadric surface. Finally, we extend Monnier's investigation of the image of the global signature homomorphism for surfaces.
\end{abstract}

\maketitle

\section{Introduction}

By convention, an algebraic variety over a field $k$ is a separated $k$-scheme of finite type. Let $X$ be a smooth algebraic variety over $k$ with $2\in k^\times$ and let $\Leu$ be a line bundle on $X$. With these data, we may associate the \emph{Witt groups} $\Wr^i(X,\Leu)$ of $X$ twisted by $\Leu$ \cite{balmerTriangularWittGroups2000}. These groups, which are functorial in $X$ and only depend on the residue of $i$ mod~$4$, are analogues of the $\Kr$-theory group $\Kr_0(X)$ of $X$ for vector bundles (or, more accurately, perfect complexes) equipped with additional bilinear data encoded by the abstract notion of duality \cite{balmerTriangularWittGroups2000}. Thus the group $\Wr(X,\Leu)=\Wr^0(X,\Leu)$ deals with symmetric bilinear forms, while $\Wr^2(X,\Leu)$ controls symplectic forms, and the groups $\Wr^i(X,\Leu)$ for odd $i$ govern the theory of formations in the sense of \cite{ranickiAlgebraicTheoryFoundations1973} (see \cite[Theorems 6.1, 7.1 and 8.1]{walterGrothendieckWittGroupsTriangulated} in the closely related context of Grothendieck--Witt groups).

Although Hermitian $\Kr$-theory, the cohomology theory underlying the groups $\Wr^i(X,\Leu)$, as developed, \emph{e.g.}, in \cite{schlichtingHermitianKtheoryDerived2017}, is similar in many ways to algebraic $\Kr$-theory, the manipulation of Witt groups is technically more delicate than that of the $\Kr$-theory groups. The conceptual reason is that Hermitian $\Kr$-theory is not an orientable cohomology theory, hence the presence of the twisting line bundle $\Leu$. This twist is essential in the use of \emph{Gysin morphisms} or pushforward along proper maps. If $f\colon Y\to X$ is a proper morphism of smooth $k$-schemes, then there is an induced pushforward map \[f_*\colon\Wr^i(Y,\omega_{f}\otimes f^*\Leu)\to\Wr^{i-\delta}(X,\Leu)\] of Grothendieck--Witt groups, where $\omega_{f}$ is the orientation sheaf of $f$ and $\delta=\dim(Y)-\dim(X)$ is its relative dimension (see \cite{calmesPushforwardsWittGroups2011} in the context of Witt groups). Thus the elucidation of the twisted groups $\Wr^i(X,\Leu)$ is essential to take full advantage of the cohomological structure of Witt groups.

Now suppose that $k=\Rb$ is the field of real numbers. In this situation, Witt groups have been studied by M. Knebusch \cite{knebuschAlgebraicCurvesReal1976a} and J.-P. Monnier \cite{monnierWittGroupTorsion2002} for arbitrary smooth curves, and by R. Sujatha \cite{sujathaWittGroupsReal1990} for smooth projective surfaces, with additional literature in special cases such as the work of Sujatha and J. van Hamel \cite{sujathaLevelWittGroups2000} on Enriques surfaces. In \cite{karoubiWittGroupReal2016,karoubiWittGroupReal2020}, Witt groups of real algebraic varieties were studied via a topological counterpart, the \emph{Real} (symmetric and symplectic) Witt group $\mathrm{WR}$ (here, Real is in the sense of M. Atiyah \cite{atiyahTHEORYREALITY1966}), which is an invariant of spaces with involution; in \cite[Section 1]{karoubiWittGroupReal2020} in particular, the authors revisit Sujatha's calculations to contextualise their work. 
These computations only concern the untwisted Witt groups and, usually, the Witt group of symmetric bilinear forms specifically. The description of Witt groups obtained in the aforementioned articles depend essentially on the topology of the real locus and on geometric data, such as the genus for curves; the required information is in general contained in the mod $2$ étale cohomology of the variety under consideration and that of its complexification.

Such a topological description was also obtained for the untwisted \emph{symplectic} Witt group $\Wr^2$ in special cases by J. Barge and M. Ojanguren in \cite{bargeFibresAlgebriquesSurface1987}, and the main result of this article is a blueprint for the type of assertions that we seek. In \cite{bargeFibresAlgebriquesSurface1987}, the authors study the classification of vector bundles on a smooth real affine surface $X$. General arguments reduce this classification to the study of rank $2$ bundles. If such a bundle $E$ is further orientable (that is, if $E$ has trivial determinant line bundle), then the exterior product $x\otimes y\mapsto x\wedge y$ induces a symplectic form $\varphi_E\colon E\otimes E\to\det E\simeq\Osc$ and thus an element $[E,\varphi_E]$ of the Grothendieck--Witt group $\GW^2(X)$ of symplectic forms. The map taking the isomorphism class $\{E\}$ of an orientable rank $2$ bundle $E$ to $[E,\varphi_E]-[\Osc^2,\varphi_{\Osc^2}]$ is then a bijection from the set $\Vsc_2^o(X)$ of isomorphism classes of such bundles to the subgroup of $\GW^2(X)$ composed of rank $0$ classes. Thus the group $\GW^2(X)$ contains all the information required to understand the theory of orientable vector bundles on smooth real affine surfaces. According to the \emph{Karoubi periodicity} exact sequence, the group $\GW^2(X)$ sits in an exact sequence \[\Kr_0(X)\xrightarrow{H}\GW^2(X)\to\Wr^2(X)\to 0,\] where $H$ denotes the symplectic hyperbolic form homomorphism. The main result \cite[Théorème 6.2]{bargeFibresAlgebriquesSurface1987} of Barge--Ojanguren's work identifies $\Wr^2(X)$ in terms of essentially topological information on the real locus $X(\Rb)$ of the smooth affine surface $X$. The following theorem is then a direct generalisation of this result and, in essence, of its proof.

\begin{theonn}[Corollary \ref{cor:small_d_w_d}]
Let $X$ be a smooth algebraic variety over $\Rb$ of dimension $d\leq 3$ and let $\Leu$ be a line bundle on $X$. Let $d_L\colon \Hr^{d-1}(X(\Rb),\Zb/2)\to\Hr^{d}(X(\Rb),\Zb(\Leu))$ be the connecting homomorphism in the cohomology long exact sequence associated with the epimorphism $\Zb(\Leu)\to\Zb/2$ of sheaves, where $\Zb(\Leu)$ is the local system determined by the real topological line bundle $\Leu(\Rb)$. If $X$ is not proper or $X(\Rb)$ is not empty, there is a canonical isomorphism \[\Wr^d(X,\Leu)\cong\Hr^d(X(\Rb),\Zb(\Leu))/d_\Leu\Hr_\alg^{d-1}(X(\Rb),\Zb/2)\] of abelian groups, where $\Hr_\alg^{d-1}(X(\Rb),\Zb/2)$ is the subgroup of $\Hr^{d-1}(X(\Rb),\Zb/2)$ generated by the mod $2$ fundamental classes of algebraic curves on $X$.
\end{theonn}

The group $\Wr^d(X,\Leu)$ can also be described in terms of cohomology operations on the mod $2$ Chow groups when $X$ is proper and $X(\Rb)$ is empty.

If $X$ is a smooth variety of dimension $d\leq 3$, inspection of the Gersten--Witt spectral sequence shows that $\Wr^i(X,\text{--})=0$ if the remainder of $i$ mod $4$ does not lie in $\{0,\ldots,d\}$. Hence the above theorem only leaves the group $\Wr^0(X,\Leu)=\Wr(X,\Leu)$ to elucidate for curves. Here again, the answer is topological up to $2$-torsion, and the $2$-torsion information is of geometric nature in that it can be computed after extension of scalars to $\Cb$ once the topology is known. Call a closed point $y$ of an $\Rb$-variety complex if its residue field $\kappa(y)$ is isomorphic to $\Cb$.

\begin{theonn}[Theorem \ref{theo:twisted_W_0_curve}]
Let $X$ be a smooth algebraic curve over $\Rb$ and let $\Leu$ be a line bundle on $X$. Let $Y$ be a smooth compactification of $X$; denote by $c$ the number of complex points of $Y\setminus X$ and by $g$ the genus of $Y_\Cb=Y\times_\Rb\Spec\Cb$. If $X$ is not proper or $X(\Rb)$ is not empty, then there is a split exact sequence \[0\to(\Zb/2)^{g+c}\to\Wr(X,\Leu)\xrightarrow{\gamma^0}\Hr^0(X(\Rb),\Zb(\Leu))\] of abelian groups induced by the global signature $\gamma^0$. If $X$ is proper and $X(\Rb)$ is empty and $\Leu$ is not a square, then $\Wr(X,\Leu)\simeq(\Zb/2)^{g+1}$.
\end{theonn}

Using the methods presented below, we could also obtain a description of the group $\Wr(X,\Leu)$ when $\Leu$ is a square, the curve $X$ is proper and $X(\Rb)$ is empty, but this was already done by Monnier \cite[Theorem 2.9]{monnierWittGroupTorsion2002}; this computation was also recovered by different methods by Karoubi--Schlichting--Weibel \cite[Theorem 4.6]{karoubiWittGroupReal2016}. We also note that the image of the global signature $\gamma^0$ was computed in \cite[Proposition 5.6]{lerbetImageHigherSignature2026} (including in the twisted case) so the above theorem gives an essentially complete description of $\Wr(X,\Leu)$.
%
%

We finally turn to surfaces, for which, as noted before, Corollary \ref{cor:small_d_w_d} only leaves $\Wr^0(X,\Leu)$ and $\Wr^1(X,\Leu)$ to investigate. This involves the $\Ibf^\star$-cohomology groups $\Hr^*(X,\Ibf_\Leu^\star)$ of $X$, for which we refer to Section \ref{section:preliminaries} for more details.

\begin{theonn}[Proposition \ref{prop:exact_sequence_twisted_I} and Remark \ref{rema:computation_twisted_I_étale_cohomology}, Theorem \ref{theo:W0_twisted_proj_surface}, Corollary \ref{cor:W1_surface}]
Let $X$ be a smooth real surface.
\begin{enumerate}
	\item Let $\Leu$ be a line bundle on $X$. If $\Leu$ is not a square, then there is an exact sequence \[0\to\Hr^0(X,\Ibf^2_\Leu)\to\Wr(X,\Leu)\to\Hr_\et^1(X,\Zb/2)\xrightarrow{\cup c_1^\et(\Leu)}\Hr_\et^3(X,\Zb/2)\] where $c_1^\et(\Leu)\in\Hr_\et^2(X,\Zb/2)$ is the étale Chern class of $\Leu$; moreover, the torsion subgroup of $\Wr(X,\Leu)$ is determined as an abelian group by the mod~$2$ étale cohomology of $X$ and $X\times_\Rb\Spec\Cb$ and the homomorphism $\cup c_1^\et(\Leu):\Hr_\et^1(X,\Zb/2)\to\Hr_\et^3(X,\Zb/2)$.
	\item There is an exact sequence \[0\to\Coker\partial_{X,\Leu}^{0,1}\to\Wr^1(X,\Leu)\to\frac{\Ker\Sq_\Leu^2}{\Zb/2\cdot\overline{c}_1(\Leu)}\to 0\] where $\partial_{X,\Leu}^{0,1}:\Hr^0(X,\bar{\Ibf})\to\Hr^1(X,\Ibf_\Leu^2)$ is the connecting homomorphism for the epimorphism $\Ibf(\Leu)\to\bar{\Ibf}$ and $\Sq_\Leu^2:\CH^1(X)/2\to\CH^2(X)/2$ carries $x$ to $x^2+x\cdot\overline{c}_1(\Leu)$ (here $\cdot$ is the intersection product on Chow groups).
\end{enumerate}
\end{theonn}

\vspace{-4pt}

Contrary to what happens for the $2$-shifted Witt group $\Wr^2(X,\Leu)$ (determined by Corollary \ref{cor:small_d_w_d} for surfaces), even forgetting the contribution coming from the (twisted) squaring operation on $\CH^1(X)/2$, the group $\Wr^1(X,\Leu)$ cannot be described in terms of $\Hr^1(X(\Rb),\Zb(\Leu(\Rb)))$ and a subgroup of ``algebraic'' cycles in $\Hr^0(X(\Rb),\Zb/2)$ (Remark \ref{rema:no_analogue_W1_surfaces}) and seems significantly more delicate to access in general. Let us however mention that the group $\Hr^1(X,\Ibf^2_\Leu)$ appearing in the description of $\Wr^1(X,\Leu)$ in item 3. in the above statement is reasonably explicit: it is the direct sum of $\Hr^1(X(\Rb),\Zb(\Leu(\Rb)))$ and of a $\Zb/2$-vector space whose dimension is determined by mod $2$ Betti numbers of $X(\Rb)$ and the mod $2$ étale cohomology group $\Hr_\et^3(X,\Zb/2)$ (Proposition \ref{prop:H1I2_surface}). We apply the above results to the calculation the twisted and shifted Witt groups of the smooth real anisotropic quadric surface; see Theorem \ref{theo:witt_groups_anisotropic_quadric}.


\subsubsection*{Contents}

The organisation of the article is as follows. Section \ref{section:preliminaries} collects preliminary notions and results for use in the sequel. Section \ref{section:top_witt} is dedicated to the proof of Corollary \ref{cor:small_d_w_d} on the top Witt group in small dimension. Section \ref{section:curves} discusses the twisted and shifted Witt groups of real curves, and Section \ref{section:surfaces} the case of surfaces, generalising in particular \cite{sujathaWittGroupsReal1990} for the unshifted Witt group. We also systematise and extend computations of the image of the global signature for surfaces due to Monnier \cite{monnierImageTotalSignature1997, monnierUnramifiedCohomologyQuadratic2000}.

\subsection*{Ackwnowledgements}

The author thanks Olivier Benoist for useful suggestions and discussions, particularly around Example \ref{exe:connecting_homomorphism_nontrivial}. Part of this research was carried out while he was a graduate student and he thanks his advisor Jean Fasel for numerous formative discussions around this material as well as for helpful comments. He was supported by ANR project CYCLADES, grant number ANR-23-CE40-0011, during the later stages of this project.

\section{Preliminaries}\label{section:preliminaries}

The letter $k$ denotes a perfect field of characteristic not $2$. The reader may safely take $k$ to be of characteristic $0$ in view of our eventual applications.

\subsection{Chow groups}

We denote by $\CH^*(X)$ the total Chow ring of algebraic cycles modulo rational equivalence on a smooth $k$-variety $X$. The product $\cdot$ is the intersection product, which we usually do not denote explicitly, writing $ab$ for $a\cdot b$.

\subsection{\'Etale cohomology}

We will use étale cohomology $\Hr_\et^*(\text{--},\Fsc)$ with coefficients in a torsion sheaf $\Fsc$ as presented, \emph{e.g.}, in \cite{milneEtaleCohomology1980}. We usually do not explicitly denote the cup-product in étale cohomology, writing $ab$ for $a\cup b$. \emph{In this paper, we mostly use étale cohomology with coefficients in $\Zb/2$.} For simplicity, we suppress $\Zb/2$-coefficients from the notation, writing $\Hr_\et^*(X)$ for $\Hr_\et^*(X,\Zb/2)$. Observe that $\Hr_\et^1(k)=k^\times/k^{\times 2}$ by the Kummer exact sequence: we denote by $(a)$ the class of $a\in k^\times$ in this group and set $\omega=(-1)$. If $X$ is any $k$-scheme, we also denote by $\omega$ the pullback of this latter class in $\Hr_\et^1(X)$ unless confusion arises from this practice.

\subsection{Witt groups}

Let $X$ be a smooth algebraic variety over~$k$ and let $\Leu$ be a line bundle on $X$. We denote by $\Wr^i(X,\Leu)$ the $i$-shifted $\Leu$-twisted Witt group of perfect complexes on $X$ in the sense of \cite{balmerTriangularWittGroups2000}. These groups only depend on the residue of~$i$ mod~$4$ up to canonical isomorphism, and on the class of $\Leu$ in $\Pic(X)/2$ up to isomorphism (see \cite[(1), 2.1 Notation]{balmerBasesTotalWitt2012}): consequently, we omit $\Leu$ from the notation when $\Leu$ is a square, and we omit $i$ when it is congruent to $0$ mod $4$.

The group $\Wr(X,\Leu)$ is the Witt group of $\Leu$-valued symmetric bilinear forms on $X$ and the tensor product of such forms induces an operation $\Wr(X,\Leu)\times\Wr(X,\Leu')\to\Wr(X,\Leu\otimes\Leu')$. This turns $\Wr(X)$ into a ring whose unit is the class of the bilinear form $\langle 1\rangle\colon x\otimes y\mapsto xy$ on $\Osc_X$, and $\Wr(X,\Leu)$ into a module over $\Wr(X)$. These remarks also apply to the Witt groups of symmetric forms. We denote by $\Ir(X,\Leu)$ the subgroup of $\Wr(X,\Leu)$ of even rank forms. When $\Leu$ is a square, the subgroup $\Ir(X)$ of $\Wr(X)$ is further an ideal called the fundamental ideal of $\Wr(X)$. We denote the $n$-th power of this ideal by $\Ir^n(X)$, and for every $n\leq 0$, we set $\Ir^n(X)=\Wr(X)$. Finally, we set $\bar{\Ir}^n(X)=\Ir^n(X)/\Ir^{n+1}(X)$ for every $n$, so that $\overline{\Wr}(X)=\Wr(X)/\Ir(X)=\Hr^0(X,\Zb/2)$ and $\bar{\Ir}^n(X)=0$ for $n<0$.

The groups $\Wr^i(X,\Leu)$ are best understood when $X$ is the spectrum of a field. Indeed, if $F$ is any field and $L$ is an $F$-vector space of dimension $1$, one has $\Wr^i(F,L)=0$ for $i\neq 0$ mod $4$ \cite[Proposition 5.2]{balmerGerstenWittSpectral2002}. Conversely, the Witt group of fields control those of general smooth varieties via the Gersten--Witt spectral sequence (\cite{balmerGerstenWittSpectral2002}), which takes the following form: 
\begin{equation}\label{eq:GW}
\Er_1^{p,q}=\bigoplus_{x\in X^{(p)}}\Wr^q(\kappa(x),\omega_{x/X}\otimes\Leu(x))\Rightarrow\Wr^{p+q}(X,\Leu)
\end{equation}
In these expressions, we denote by $X^{(p)}$ the set of codimension $p$ points of $X$, by $\omega_{x/X}$ the orientation sheaf of the inclusion $x\hookrightarrow X$ and by $\Leu(x)$ the fibre of $\Leu$ at $x$. The line $q=q_0$ of (\ref{eq:GW}) vanishes if $q_0\neq 0$ mod $4$, and is independent of $q_0$ if $q_0=0$ mod $4$. Consequently, if $X$ has dimension $d\leq 3$, the Gersten--Witt spectral sequence collapses and yields an isomorphism $\Wr^p(X,\Leu)\cong\Er_2^{p,0}$ for every $p\in\{0,\ldots,3\}$. Since $X^{(p)}$ is empty for $i>d$, one has $\Er_1^{p,0}=0$ for any such $i$ and it follows that $\Wr^p(X,\Leu)=0$ if $d+1\leq p\leq 3$ (in the context of the present article, the vanishing of, \emph{e.g.}, the symplectic Witt group with trivial twist for possibly singular real curves also follows from the topological arguments of \cite{karoubiWittGroupReal2016}, see \cite[Proposition 4.8]{karoubiWittGroupReal2016}).

We denote by $X_{\Zar}$ the small Zariski site of $X$ whose objects are the open subsets of $X$. We denote by $\Wbf(\Leu)$ the sheaf on $X_\Zar$ associated with the presheaf $U\mapsto\Wr(U,\Leu_{|U})$. The lines $q=0$ mod $4$ of the first page of (\ref{eq:GW}), denoted by $\Cr(X,\Wbf(\Leu))$, then compute the cohomology of $X$ with coefficients in $\Wbf(\Leu)$ so $\Er_2^{p,0}=\Hr^p(X,\Wbf(\Leu))$. In particular, if $X$ has dimension $\leq 3$, this yields an isomorphism $\Wr^p(X,\Leu)\cong\Hr^p(X,\Wbf(\Leu))$ for every $p\in\{0,\ldots,3\}$. Moreover, let $\Ibf^n(\Leu)$ be the sheaf on $X_\Zar$ associated with the presheaf $U\mapsto\Ir^n(U)\Wr(U,\Leu)$ (for ease of notation, we also write $\Ibf_\Leu^n$, and thus $\Wbf_\Leu$ for $n\leq 0$, for this sheaf). The maps in the line $q=0$ mod $4$ of the Gersten--Witt spectral sequence are compatible with the powers of the fundamental ideal (\cite[Théorème 9.2.4]{faselGroupesChowWitt2008}). Thus there is an induced complex $\Cr(X,\Ibf^n(\Leu))$ of abelian groups whose degree $p$ term is given by \[\bigoplus_{x\in X^{(p)}}\Ir^{n-p}(\kappa(x),\omega_{x/X}\otimes\Leu(x)),\] whose cohomology groups compute $\Hr^*(X,\Ibf^n(\Leu))$. The inclusions $\Ir^{n+1}(\text{--})\subseteq\Ir^n(\text{--})$ induce an inclusion $\Cr(X,\Ibf^{n+1}_\Leu)\subseteq\Cr(X,\Ibf^n_\Leu)$ of complexes; the quotient $\Cr(X,\bar{\Ibf}^n)$ has groups of the form \[\bigoplus_{x\in X^{(p)}}\bar{\Ir}^{n-p}(\kappa(x))\] (here we use the canonical isomorphism $\bar{\Ir}^{t}(F)\cong\bar{\Ir}^t(F,L)$ of \cite[Lemme E.1.3]{faselGroupesChowWitt2008}), and its cohomology computes $\Hr^*(X,\bar{\Ibf}^n)$ where $\bar{\Ibf}^n=\Ibf^n/\Ibf^{n+1}$.

\subsection{The Bloch--Ogus spectral sequence} 

Let $X$ be a smooth $k$-variety. There is a Bloch--Ogus, or coniveau, spectral sequence:
\begin{equation}\label{eq:bloch_ogus}
\Er_1^{p,q}=\bigoplus_{x\in X^{(p)}}\Hr_\et^{q-p}(\kappa(x))\Rightarrow\Hr_\et^{p+q}(X)
\end{equation}
studied in \cite{blochGerstensConjectureHomology1974} and converging to étale cohomology with coefficients in $\Zb/2$ filtered by codimension of the support. Letting $\Hsc^n$ be the sheaf on $X_\Zar$ associated with the presheaf $U\mapsto\Hr_\et^n(U)$, the line $q=n$ of the $\Er_1$-page of (\ref{eq:bloch_ogus}) is a complex $\Cr(X,\Hsc^n)$ of abelian groups computing the cohomology of $X$ with coefficients in $\Hsc^n$; this is a special case of the main result of \cite{blochGerstensConjectureHomology1974}.

The affirmation of the Milnor conjecture (\cite{voevodskyMotivicCohomology2coefficients2003,orlovExactSequenceKM22007}) yields a natural isomorphism $\Hr_\et^*(F)\cong\bar{\Ir}^*(F)$ of graded rings for every field $F$ of characteristic not~$2$, carrying $(a)\in\Hr_\et^1(F)$ to the class mod $\Ir^2(F)$ of the diagonal rank $2$ form $\llangle a\rrangle=\langle 1,-a\rangle$ in degree $1$. In particular, the class $\omega$ corresponds to $\llangle -1\rrangle=\langle 1,1\rangle$ under this isomorphism and we sometimes abuse notation by overloading the letter~$\omega$ and using it for the class of $\llangle -1\rrangle$. It follows that there is an identification $\Cr(X,\Hsc^n)\cong\Cr(X,\bar{\Ibf}^n)$ of complexes of abelian groups, in particular an identification $\Hsc^n\cong\bar{\Ibf}^n$ of sheaves on $X_\Zar$. This allows one to control the cohomology of $X$ with coefficients in $\bar{\Ibf}^*$ in terms of the étale cohomology of $X$.

\begin{rema}
The Bloch--Ogus theorem yields an identification of $\Hr^n(X,\Hsc^n)$ with $\Ch^n(X)$. Thus in view of the above isomorphism $\bar{\Ibf}^*\cong\Hsc^*$, we obtain an identification $\Hr^n(X,\bar{\Ibf}^n)\cong\Ch^n(X)$.
\end{rema}

An important structural feature of (\ref{eq:bloch_ogus}) is that since étale cohomology vanishes in negative degrees, one has $\Er_1^{p,q}=0$ for $p>q$. Thus $\Hr^p(X,\Hsc^q)=0$ for $p>q$ and the differentials $d_r^{p,q}$ vanish for every $r\geq 2$ if $p\geq q-1$. It follows that if $p\geq q-1$, then $\Er_\infty^{p,q}$ is a quotient of $\Hr^p(X,\Hsc^q)$. This yields an edge homomorphism \[\Hr^{p}(X,\Hsc^q)\to\Hr_\et^{p+q}(X)\] for $p\geq q-1$. Since $\Ch^n(X)=\Hr^n(X,\Hsc^n)$, this observation yields an edge homomorphism \[\Ch^n(X)\twoheadrightarrow\Er_{\infty}^{n,n}\hookrightarrow\Hr_\et^{2n}(X).\] We call it the \emph{étale cycle class map} and denote it by $\gamma_\et^n$. It coincides with the one of \cite[Chapter VI, §9]{milneEtaleCohomology1980} and is therefore compatible with products, pullbacks and proper pushforwards. By inspection of (\ref{eq:bloch_ogus}), we then see that $\gamma_\et^1:\Ch^1(X)\to\Hr_\et^2(X)$ is injective; modulo the isomorphism $\Pic(X)/2\cong\Ch^1(X)$ induced by $c_1$, it takes the square class of a line bundle $\Leu$ to the étale Chern class $c_1^\et(\Leu)$ of $\Leu$.

\subsection{General remarks on $\Ibf^*$-cohomology}


\subsubsection{Steenrod operations and $\Ibf^*$-cohomology}

We denote by $\Sq^2:\Ch^*\to\Ch^{*+1}$ the Steenrod squares on Chow groups mod $2$ of Brosnan \cite{brosnanSteenrodOperationsChow2003} and Voevodsky \cite{voevodskyReducedPowerOperations2003} (this notation is not Brosnan's, who writes $\Sq^1$ for this map, but it is more compatible with the indexing in motivic cohomology). Let $X$ be a smooth variety over $k$ and let $\Leu$ be a line bundle on $X$. Given $i,j\in\Zb$, we then denote by $\partial_{X,\Leu}^{i,j}:\Hr^i(X,\bar{\Ibf}^j)\to\Hr^{i+1}(X,\Ibf^{j+1}_\Leu)$ the connecting homomorphism in the long exact sequence associated with the epimorphism $\Ibf^j(\Leu)\to\bar{\Ibf}^j$ and by $\pi_{X,\Leu}^{i,j}\Leu^{i,j}:\Hr^i(X,\Ibf^j_\Leu)\to\Hr^i(X,\bar{\Ibf}^j)$ the map induced in cohomology in degree $i$ by this epimorphism. Following \cite{asokSecondaryCharacteristicClasses2015}, we denote by $\Phi_{X,\Leu}^{i,j}$ the composite \[\Phi_{X,\Leu}^{i,j}:\Hr^i(X,\bar{\Ibf}^j)\xrightarrow{\partial_{X,\Leu}^{i,j}}\Hr^{i+1}(X,\Ibf^{j+1}_\Leu)\xrightarrow{\pi_{X,\Leu}^{i+1,j+1}}\Hr^{i+1}(X,\bar{\Ibf}^{j+1}).\] We omit $\Leu$ from the notation if $\Leu$ is a square, and $X$ whenever there is no risk of confusion.

\begin{theo}[Totaro]\label{theo:Totaro}
Let $i\geq 0$ and let $X$ be a smooth $k$-variety. Then $\Phi^{i,i+1}=\Sq^2$ as maps from $\Ch^i(X)$ to $\Ch^{i+1}(X)$.
\end{theo}

\begin{proof}
See \cite[Theorem 1.1]{totaroNONINJECTIVITYMAPWITT2003}.
\end{proof}

%

\begin{theo}[Asok--Fasel]\label{theo:twisted_differential_pardon}
Let $X$ be a smooth variety over $k$. If $\Leu$ is a line bundle on $X$, then for any $i,j\geq 0$, the morphism $\Phi_{X,\Leu}^{i,j}$ is the sum of $\Phi_X^{i,j}$ and of product with the mod $2$ Chern class $\bar{c}_1(\Leu)\in\Ch^1(X)=\Hr^1(X,\bar{\Ibf}^1)$ of $\Leu$.
\end{theo}

\begin{proof}
See \cite[Theorem 3.4.1]{asokSecondaryCharacteristicClasses2015}.
\end{proof}

\begin{rema}\label{rema:twisted_steenrod_square}
In particular, one has $\Phi_{X,\Leu}^{i,i}=\Sq^2+\bar{c}_1(\Leu)\cdot(\text{--})$ for any $i$. We also denote this homomorphism by $\Sq^2_\Leu$ and call it the \emph{twisted} Steenrod square. We sometimes write $\Sq_x^2$ for the map $y\mapsto\Sq^2(y)+xy$, so that $\Sq_\Leu^2=\Sq_{\overline{c}_1(\Leu)}^2$.
\end{rema}

\subsubsection{Reduction to the stable range} 

The aim of the following two lemmas is to reduce the computation of $\Ibf^*$-cohomology to the case of higher powers of the fundamental ideal, which are closer to the ``stable range'' (above the dimension) and thus easier to understand.

\begin{lem}\label{lem:easy_higher_pardon}
Let $X$ be a smooth $k$-variety and let $\Leu$ be a line bundle on $X$. Let $s\geq r\geq 0$ be integers. Then $\Hr^{i}(X,\Ibf^r_\Leu/\Ibf^s_\Leu)=0$ for every $i\geq s$. Thus the inclusion $\Ibf^s(\Leu)\subseteq\Ibf^r(\Leu)$ induces an isomorphism $\Hr^i(X,\Ibf^s_\Leu)\cong\Hr^i(X,\Ibf^r_\Leu)$ for every $i\geq s+1$.
\end{lem}

\begin{proof}
The first statement implies the second by examination of the cohomology long exact sequence associated with the inclusion $\Ibf^s(\Leu)\subseteq\Ibf^r(\Leu)$. Now this inclusion induces an inclusion $\Cr(X,\Ibf^s_\Leu)\subseteq\Cr(X,\Ibf^r_\Leu)$ of complexes and the cohomology of the cokernel $\Cr(X,\Ibf^r_\Leu/\Ibf^s_\Leu)$ of this inclusion computes $\Hr^*(X,\Ibf_\Leu^r/\Ibf^s_\Leu)$. But in fact, for every $i\geq s$, the degree $i$ term of the former complexes is given by \[\Cr(X,\Ibf^s_\Leu)^i=\Cr(X,\Ibf^r_\Leu)^i=\bigoplus_{x\in X^{(i)}}\Wr(\kappa(x),\omega_{x/X}\otimes\Leu(x)).\] Therefore $\Cr(X,\Ibf^r_\Leu/\Ibf^s_\Leu)^i=0$ for every $i\geq s$ and thus $\Hr^i(X,\Ibf^r_\Leu/\Ibf^s_\Leu)=0$, as required.
\end{proof}

Let $X$ be a smooth $k$-variety and let $\Leu$ be a line bundle on $X$. We consider the connecting homomorphism $\partial_{X,\Leu}^{0,0}:\Hr^0(X,\bar{\Wbf})\to\Hr^1(X,\Ibf_\Leu)$ in the cohomology long exact sequence for the epimorphism $\Wbf(\Leu)\to\bar{\Wbf}$ of sheaves on $X$. We let $e(\Leu)$ denote the image of $\langle 1\rangle\in\Hr^0(X,\bar{\Wbf})$ under $\partial_{X,\Leu}^{0,0}$; this notation is justified by the fact that according to \cite[Lemma 3.1]{faselProjectiveBundleTheorem2013}, this cohomology class is indeed the ($\Ibf$-cohomological) Euler class of $\Leu$ in the sense of \cite[§3]{faselProjectiveBundleTheorem2013}.

\begin{lem}\label{lem:baby_pardon}
The inclusion $\Ibf(\Leu)\to\Wbf(\Leu)$ induces an isomorphism $\Hr^1(X,\Wbf_\Leu)\cong\Hr^1(X,\Ibf_\Leu)/e(\Leu)$. Thus the map $\Hr^1(X,\Ibf)\to\Hr^1(X,\Wbf)$ is an isomorphism and if $\Leu$ is not a square, then $\Hr^1(X,\Wbf_\Leu)=\Hr^1(X,\Ibf_\Leu)/\Zb/2$ and the inclusion $\Hr^0(X,\Ibf_\Leu)\subseteq\Hr^0(X,\Wbf_\Leu)$ is surjective.
\end{lem}

\begin{proof}
Note that $e(\Leu)$, being the image under $\partial_{X,\Leu}^{0,0}$ of the $2$-torsion element $\langle 1\rangle$ of $\Hr^0(X,\bar{\Wbf})=\Hr^0(X,\Zb/2)$, is $2$-torsion. In view of the cohomology exact sequence 
\[\begin{tikzcd}
	0 & {\Hr^0(X,\Ibf_\Leu)} & {\Hr^0(X,\Wbf_\Leu)} & {\Hr^0(X,\Wbf/\Ibf)} & \\
	& {\Hr^1(X,\Ibf_\Leu)} & {\Hr^0(X,\Wbf_\Leu)} & {\Hr^1(X,\Wbf/\Ibf)} & 
	\arrow[from=1-1, to=1-2]
	\arrow[from=1-2, to=1-3]
	\arrow[from=1-3, to=1-4]
	\arrow[from=1-4, to=2-2]
	\arrow[from=2-2, to=2-3]
	\arrow[from=2-3, to=2-4]
\end{tikzcd}\]
and since $\Hr^p(X,\Wbf/\Ibf)=0$ for $p>0$ by the previous lemma, the map $\Hr^1(X,\Ibf_\Leu)\to\Hr^1(X,\Wbf_\Leu)$ induces an isomorphism $\Hr^1(X,\Wbf_\Leu)\cong\Hr^1(X,\Ibf_\Leu)/\Zb/2\cdot e(\Leu)$ as in the statement of the lemma. By \cite[Example 3.3.1]{fultonIntersectionTheory1998} (and \cite[Lemma 3.1]{faselProjectiveBundleTheorem2013}), the image of $e(\Leu)$ under the reduction homomorphism $\Hr^1(X,\Ibf_\Leu)\to\Hr^1(X,\bar{\Ibf})=\Ch^1(X)$ is the mod $2$ first Chern class $\bar{c}_1(\Leu)$ of $\Leu$. Therefore, if $\Leu$ is a square, then $e(\Leu)=0$ and thus $\Hr^1(X,\Ibf)=\Hr^1(X,\Wbf)$. If $\Leu$ is not a square, then $e(\Leu)$ is nonzero so the map $\partial_{X,\Leu}^{0,0}:\Zb/2\cdot\langle 1\rangle\to\Hr^1(X,\Ibf_\Leu)$, having image generated by $e(\Leu)$, is injective. Hence the map $\Hr^0(X,\Ibf_\Leu)\to\Hr^0(X,\Wbf_\Leu)$, whose cokernel is $\Ker\partial_{X,\Leu}^{0,0}$, is onto.
\end{proof}

\subsubsection{A lemma}

We learned the following fact from \cite[proof of Lemma 2.1]{sujathaWittGroupsReal1990}.

\begin{lem}\label{lem:Sujatha}
Let $X$ be a connected smooth variety over $k$. The homomorphism $\pi_X^{0,1}:\Hr^0(X,\Ibf)\to\Hr^0(X,\bar{\Ibf})$ is then surjective.
\end{lem}

\begin{proof}
See \cite[§1.6]{parimalaGradedWittRing1992}.
\end{proof}

\begin{rema}\label{rema:steenrod_trivial}
In particular, the connecting map $\partial_X^{0,1}$ is trivial; \emph{a fortiori}, so is $\Phi_X^{0,1}$.
\end{rema}

\subsection{The cohomology of real algebraic varieties}

In the sequel, if $X$ is an $\Rb$-scheme, we denote by $X_\Cb$ its complexification $X_\Cb=X\times_\Rb\Spec\Cb$.

\subsubsection{A hypothesis on real algebraic varieties}

We introduce a notation for the following hypothesis on smooth $\Rb$-varieties for ease of future reference.

\begin{itemize}
	\item[$(*)$] The variety $X$ is not proper or its real locus $X(\Rb)$ is not empty.
\end{itemize}

\subsubsection{The real-complex exact sequence in mod $2$ étale cohomology}

Let $X$ be a smooth $\Rb$-variety; denote by $\pi:X_\Cb=X\times_\Rb\Spec\Cb\to X$ the projection. Then there is an exact sequence \[0\to\Zb/2\to\pi_*\Zb/2\to\Zb/2\to 0\] of étale sheaves on $X$ \cite[(2.3)]{colliot-theleneZerocyclesCohomologyReal1996}. It induces an exact sequence
\begin{equation}\label{eq:real_complex_exact_sequence}
\cdots\to\Hr_\et^{p-1}(X)\xrightarrow{\partial}\Hr_\et^{p}(X)\xrightarrow{\pi^*}\Hr_\et^p(X_\Cb)\xrightarrow{\pi_*}\Hr_\et^p(X)\to\cdots
\end{equation}
of abelian groups \cite[(2.4)]{colliot-theleneZerocyclesCohomologyReal1996}. We call it the \emph{real-complex exact sequence} for mod~$2$ étale cohomology. The connecting homomorphism $\partial$ is given by cup-product with the class $\omega\in\Hr_\et^1(X)$, the map $\pi^*$ is the pullback along the morphism $\pi$ in étale cohomology and $\pi_*$ is the pushforward homomorphism (norm) for the finite map $\pi$. This exact sequence is especially useful because by Artin's comparison theorem \cite[Théorème 4.1]{artinExpXVITheoreme1973}, the mod $2$ étale cohomology of $X_\Cb$ is isomorphic to the singular cohomology of the complex manifold $X(\Cb)$ with $\Zb/2$ coefficients, an invariant for which many computational tools are available.

\begin{rema}\label{rema:real_complex_H1}
If $X$ is geometrically connected, then $\pi^*:\Hr_\et^0(X)\to\Hr_\et^0(X_\Cb)$ is surjective so (\ref{eq:real_complex_exact_sequence}) determines an exact sequence \[0\to\Zb/2\cdot\omega\to\Hr_\et^1(X)\to\Hr_\et^1(X_\Cb)\] of abelian groups. Thus if $\Hr_\et^1(X_\Cb)=0$, then $\Hr_\et^1(X)=\Zb/2\cdot\omega$.
\end{rema}

\begin{lem}\label{lem:top_étale_cohomology_proper_empty}
Let $X$ be a geometrically connected smooth algebraic variety of dimension $d$ over $\Rb$. Suppose that $X$ does not satisfy $(*)$. Then the pushforward homomorphism $\pi_*:\Hr_\et^{2d}(X_\Cb)\to\Hr_\et^{2d}(X)$ is an isomorphism.
\end{lem}

\begin{proof}
By (\ref{eq:real_complex_exact_sequence}), the cokernel of $\pi_*$ is a subgroup of $\Hr_\et^{2d+1}(X)$. By \cite[Theorem 2.3.1]{colliot-theleneZerocyclesCohomologyReal1996}, there is an isomorphism \[\Hr_\et^{2d+1}(X)\cong\bigoplus_{p\geq 0}\Hr^p(X(\Rb),\Zb/2).\] Since $X(\Rb)$ is empty, the right-hand side vanishes so $\Hr_\et^{2d+1}(X)=0$. Hence $\pi_*$ is an epimorphism. Since $X$ is proper, so is $X_\Cb$ hence $\Hr_\et^{2d}(X_\Cb)=\Zb/2$ (\emph{e.g.} by Artin's comparison theorem, since $X(\Cb)$ is compact and connected). By (\ref{eq:real_complex_exact_sequence}), there is a short exact sequence \[\Hr_\et^{2d}(X)\xrightarrow{\pi^*}\Hr_\et^{2d}(X_\Cb)=\Zb/2\xrightarrow{\pi_*}\Hr_\et^{2d}(X),\] thus $\Hr_\et^{2d}(X)$ is nonzero. Thus $\pi_*$ is an epimorphism from $\Zb/2$ to a nonzero group, therefore it is an isomorphism.
\end{proof}

\begin{cor}\label{cor:end_real_complex_proper_empty}
Assume further in Lemma \ref{lem:top_étale_cohomology_proper_empty} that $\Hr_\et^{2d-1}(X_\Cb)=0$. Then $\pi^*:\Hr_\et^{2d}(X)\to\Hr_\et^{2d}(X_\Cb)$ vanishes and $\cup\omega:\Hr_\et^{2d-1}(X)\to\Hr_\et^{2d}(X)=\Zb/2$ is an isomorphism.
\end{cor}

\begin{proof}
This is clear from Lemma \ref{lem:top_étale_cohomology_proper_empty} (and its proof) and (\ref{eq:real_complex_exact_sequence}).
\end{proof}

\subsubsection{Higher signature maps}

Let $X$ be a geometrically connected smooth algebraic variety over $\Rb$. First note that if $\Fsc$ is an abelian sheaf on the real locus $X(\Rb)$ of $X$, it induces a sheaf $\iota_*\Fsc$ on $X_\Zar$ given by $\iota_*\Fsc(U)=\Fsc(U(\Rb))$. There is an induced map $\Hr^*(X,\iota_*\Fsc)\to\Hr^*(X(\Rb),\Fsc)$, which is an isomorphism if $\Fsc$ is locally constant due to results of Coste--Roy (\cite[Corollaire 3.7]{costeTopologieSpectreReel1982}, \cite[(19.2)]{scheidererRealEtaleCohomology1994}). Now let $\Leu$ be a line bundle on $X$. If $U\hookrightarrow X$ is open, there is an induced line bundle $L_{|U}$ on $U(\Rb)$. It determines a local system $\Zb(\Leu)$ of abelian groups with stalk $\Zb$ on $X(\Rb)$ (see \cite[Example 2.3]{lerbetImageHigherSignature2026} and the preceding discussion there for more details; in \cite{lerbetImageHigherSignature2026}, we wrote $\Zb(\Leu(\Rb))$ for this sheaf on $X(\Rb)$). We denote by $\iota_*\Zb(\Leu)$ the sheaf on $X_\Zar$ associated with the presheaf $U\mapsto\Hr^0(U(\Rb),\Zb(\Leu)_{|U})$ (this sheaf cohomology group coincides with the usual integral singular cohomology of $U(\Rb)$ in degree $0$ if $\Leu_{|U}$ is trivial). Using the normalised signature of symmetric bilinear forms on $X$, one may define homomorphisms $\sign_t:\Ibf^t(\Leu)\to\iota_*\Zb(\Leu)$ of abelian sheaves on $X_\Zar$ (this construction is essentially due to Jacobson \cite{jacobsonRealCohomologyPowers2017}). Using the work of Coste--Roy to identify $\Hr^*(X,\iota_*\Zb(\Leu))$ and $\Hr^*(X(\Rb),\Zb(\Leu))$, by change of sheaf of coefficients, we obtain a morphism \[\gamma_t^c(X,\Leu):\Hr^c(X,\Ibf^t_\Leu)\xrightarrow{(\sign_t)_*}\Hr^c(X(\Rb),\Zb(\Leu)).\] We omit $(X,\Leu)$ from the notation unless confusion arises from this practice. We call the morphisms $\gamma_t^c$ the \emph{quadratic real cycle class maps}; we simply write $\gamma^c$ for~$\gamma_c^c$ when $c=t$.

Multiplication by the rank $2$ form $\llangle -1\rrangle=\langle 1\rangle+\langle 1\rangle\in\Wr(\Rb)$ satisfies the equality $\sign_{t+1}\circ(\otimes\llangle -1\rrangle)=\sign_t:\Ibf^t(\Leu)\to\iota_*\Zb(\Leu)$, so that $\gamma_{t+1}^c\circ(\otimes\llangle -1\rrangle)_*=\gamma_t^c$ for all integers $c,t\geq 0$. There is a commutative ladder (\cite[(2-3)]{lerbetImageHigherSignature2026})
\begin{equation}\label{eq:ladder_normalised_signature}
\begin{tikzcd}
	0 & {\Ibf^{t+1}(\Leu)} & {\Ibf^t(\Leu)} & {\bar{\Ibf}^t} & 0 \\
	0 & {\iota_*\Zb(\Leu)} & {\iota_*\Zb(\Leu)} & {\iota_*\Zb/2} & 0
	\arrow[from=1-1, to=1-2]
	\arrow[from=1-2, to=1-3]
	\arrow["{\sign_{t+1}}"', from=1-2, to=2-2]
	\arrow[from=1-3, to=1-4]
	\arrow["{\sign_t}"', from=1-3, to=2-3]
	\arrow[from=1-4, to=1-5]
	\arrow["{\bar{\sign}_t}"', dotted, from=1-4, to=2-4]
	\arrow[from=2-1, to=2-2]
	\arrow["{\cdot 2}"', from=2-2, to=2-3]
	\arrow[from=2-3, to=2-4]
	\arrow[from=2-4, to=2-5]
\end{tikzcd}
\end{equation}
in which $\iota_*\Zb/2$ is the sheaf associated with the presheaf $U\mapsto\Hr^0(U(\Rb),\Zb/2)$ on $\Sm_\Rb$. By the previously mentioned result of Coste--Roy, we have a natural isomorphism $\Hr^*(X,\iota_*\Zb/2)\cong\Hr^*(X(\Rb),\Zb/2)$ so the map $\bar{\sign}_t$, which is by definition the morphism induced by $\sign_t$ on cokernels, induces homomorphisms $\bar{\gamma}_t^c(X):\Hr^c(X,\bar{\Ibf}^t)\to\Hr^c(X(\Rb),\Zb/2)$ (as usual, we often omit $(X)$ from the notation). Modulo the isomorphism $\bar{\Ibf}^t\cong\Hsc^t$, the map $\bar{\sign}_t$ coincides with the morphism $h_t:\Hsc^t\to\iota_*\Zb/2$ defined in \cite[§2.1]{colliot-theleneRealComponentsAlgebraic1990}. In particular, the morphism $\bar{\gamma}^c=\overline{\gamma}_c^c:\Ch^c(X)=\Hr^c(X,\overline{\Ibf}^c)\to\Hr^c(X(\Rb),\Zb/2)$ is the usual mod $2$ cycle class map for Chow groups taking a codimension $c$ cycle on $X$ to the mod $2$ fundamental class of its real locus. We refer to \cite[§2.5]{lerbetImageHigherSignature2026} for more details.

Given an integer $t\geq 0$, we denote by $\Kbf^t(\Leu)$ or by $\Kbf_\Leu^t$ the kernel of the homomorphism $\otimes\llangle -1\rrangle:\Ibf^t(\Leu)\to\Ibf^{t+1}(\Leu)$ of sheaves induced by tensor product with the diagonal form $\llangle -1\rrangle=\langle 1\rangle+\langle 1\rangle$, and by $\bar{\Kbf}^t$ the kernel of the homomorphism $\cup\omega:\Hsc^t\to\Hsc^{t+1}$ given by cup-product with the class $\omega$. The isomorphism $\bar{\Ibf}^t\cong\Hsc^t$ provided by the affirmation of the Milnor conjecture identifies the homomorphism $\otimes\llangle -1\rrangle:\bar{\Ibf}^t\to\bar{\Ibf}^{t+1}$ with the map $\cup\omega:\Hsc^t\to\Hsc^{t+1}$. In particular, there is a natural map $\Kbf^t(\Leu)\to\bar{\Kbf}^t$ between kernels induced by the quotient map $\Ibf^t(\Leu)\to\bar{\Ibf}^t\cong\Hsc^t$. The following observation eventually leads to most of the results in \cite{lerbetImageHigherSignature2026}.

\begin{lem}[\protect{\cite[Lemma 3.2]{lerbetImageHigherSignature2026}}]\label{lem:2_torsion_to_(-1)_torsion}
Let $X$ be a smooth $\Rb$-variety of dimension $d$ and let $\Leu$ be a line bundle on $X$. Then the homomorphism $\Kbf^d(\Leu)\to\bar{\Kbf}^d$ of sheaves on $X_\Zar$ is an isomorphism.
\end{lem}

Let again $X$ be a smooth $\Rb$-variety of dimension $d$ and let $\Leu$ be a line bundle on~$X$. According to results of Jacobson \cite[Corollary 8.11]{jacobsonRealCohomologyPowers2017}, for every $t>d$, the morphism $\sign_t:\Ibf^t(\Leu)\to\iota_*\Zb(\Leu)$ is an isomorphism of sheaves on $X_\Zar$. This is an integral lift of \cite[Theorem 2.3.2]{colliot-theleneRealComponentsAlgebraic1990} which asserts that $h_t:\Hsc^t\to\iota_*\Zb/2$ is an isomorphism of sheaves on $X_\Zar$ for every $t>d$. Therefore for every such $t$, the induced morphisms $\gamma_t^c$ and $\bar{\gamma}_t^c$ are isomorphisms. Jacobson's result also implies that $\Ker\sign_d=\Kbf^d(\Leu)$; moreover, the morphism $\sign_d$ is an epimorphism of sheaves on $X_\Zar$ by \cite[Lemma 3.1]{lerbetImageHigherSignature2026}, and similarly, the morphism $\bar{\sign}_d:\bar{\Ibf}^d\to\iota_*\Zb/2$ is an epimorphism \cite[Theorem 2.3.2 (b)]{colliot-theleneZerocyclesCohomologyReal1996}. The quotient morphism $\Ibf^d(\Leu)\to\bar{\Ibf}^d$ and the reduction mod $2$ map from $\iota_*\Zb(\Leu)$ to $\iota_*\Zb/2$ then fit in a commutative ladder
\begin{equation}\label{eq:commutative_ladder_signature_variety}
\begin{tikzcd}
	0 & {\Kbf^d(\Leu)} & {\Ibf^d(\Leu)} & {\iota_*\Zb(\Leu)} & 0 \\
	0 & {\bar{\Kbf}^d} & {\bar{\Ibf}^d} & {\iota_*\Zb/2} & 0
	\arrow[from=1-1, to=1-2]
	\arrow[from=1-2, to=1-3]
	\arrow[from=1-2, to=2-2]
	\arrow["{\sign_d}", from=1-3, to=1-4]
	\arrow[from=1-3, to=2-3]
	\arrow[from=1-4, to=1-5]
	\arrow["{\text{mod}\;2}", from=1-4, to=2-4]
	\arrow[from=2-1, to=2-2]
	\arrow[from=2-2, to=2-3]
	\arrow["{\bar{\sign}_d}", from=2-3, to=2-4]
	\arrow[from=2-4, to=2-5]
\end{tikzcd}
\end{equation}
of sheaves on $X_\Zar$ with exact rows. In fact, this ladder is a morphism of \emph{split} short exact sequences:

\begin{theo}[\protect{\cite[Theorem 4.1, proof of Corollary 4.3]{benoistSeparatingComponentsReal2026}}]\label{theo:sections}
There exist sections $\overline{\sigma}$ and $\sigma$ of $\overline{\sign}_d$ and $\sign_d$ respectively as morphisms of sheaves on $X_\Zar$ such that the square
\[\begin{tikzcd}
	{\Ibf^d(\Leu)} & {\iota_*\Zb(\Leu)} \\
	{\bar{\Ibf}^d} & {\iota_*\Zb/2}
	\arrow[from=1-1, to=2-1]
	\arrow["\sigma"', from=1-2, to=1-1]
	\arrow[from=1-2, to=2-2]
	\arrow["{\overline{\sigma}}", from=2-2, to=2-1]
\end{tikzcd}\]
is commutative.
\end{theo}

The proof relies on a functorial argument for a theorem of Krasnov (\cite[Corollary 3.2]{krasnovEQUIVARIANTGROTHENDIECKCOHOMOLOGY1995a}, see also \cite[Theorem 2.8]{hamelTorsionZerocyclesAbelJacobi2000}) developed in \cite[Remark 1.6 (ii)]{benoistIntegralHodgeConjecture2020}. An immediate corollary of Theorem \ref{theo:sections} is that for every $0\leq p\leq d$, there is a morphism
\begin{equation}\label{eq:commutative_ladder_signature_cohomology}
\begin{tikzcd}
	0 & {\Hr^p(X,\Kbf_\Leu^d)} & {\Hr^p(X,\Ibf_\Leu^d)} & {\Hr^p(X(\Rb),\Zb(\Leu))} & 0 \\
	0 & {\Hr^p(X,\bar{\Kbf}^d)} & {\Hr^p(X,\bar{\Ibf}^d)} & {\Hr^p(X(\Rb),\Zb/2)} & 0
	\arrow[from=1-1, to=1-2]
	\arrow[from=1-2, to=1-3]
	\arrow[from=1-2, to=2-2]
	\arrow["{\gamma_d^p}", from=1-3, to=1-4]
	\arrow["{\pi_{X,\Leu}^{p,d}}"', from=1-3, to=2-3]
	\arrow[from=1-4, to=1-5]
	\arrow["{\text{mod}\;2}"', from=1-4, to=2-4]
	\arrow[from=2-1, to=2-2]
	\arrow[from=2-2, to=2-3]
	\arrow["{\overline{\gamma}_d^p}"', from=2-3, to=2-4]
	\arrow[from=2-4, to=2-5]
\end{tikzcd}
\end{equation}
of split exact sequences in which the surjections are split by the maps induced in cohomology by the morphisms $\overline{\sigma}$ and $\sigma$ from Theorem \ref{theo:sections}.

Jacobson's results also imply that if $X(\Rb)=\emptyset$, then $\Ibf^{d+1}(\Leu)=0$ as a sheaf on $X_\Zar$ so the quotient morphism $\Ibf^d(\Leu)\to\bar{\Ibf}^d$ is an isomorphism. Thus there is an induced isomorphism
\begin{equation}\label{eq:projection_isomorphism}
\pi_{X,\Leu}^{i,d}:\Hr^{i}(X,\Ibf^{d}_\Leu)\xrightarrow{\sim}\Hr^{i}(X,\bar{\Ibf}^{d})\;\;\text{if}\;X(\Rb)=\emptyset,
\end{equation}
for any $i$. We can partially generalise the injectivity property in (\ref{eq:projection_isomorphism}). In the sequel, given an abelian group $A$ and an integer $n$, we denote by $A[n]$ the $n$-torsion subgroup of $A$.

\begin{lem}\label{lem:projection_dimension_filtration_isomorphism}
Let $X$ be a smooth $\Rb$-variety of dimension $d$ and let $\Leu$ be a line bundle on $X$. Let $i\in\{0,\ldots,d\}$. Suppose that there is an equality \[\Hr^{i}(X(\Rb),\Zb(\Leu))[4]=\Hr^{i}(X(\Rb),\Zb(\Leu))[2]\] of torsion subgroups. The homomorphism $\pi_{X,\Leu}^{i,d}:\Hr^{i}(X,\Ibf^{d}_\Leu)\to\Hr^{i}(X,\bar{\Ibf}^{d})$ is then injective on the $2$-torsion subgroup of $\Hr^{i}(X,\Ibf^{d}_\Leu)$.
\end{lem}

\begin{proof}
If $X(\Rb)$ is empty, this follows from the more precise (\ref{eq:projection_isomorphism}) so we may assume that $X(\Rb)\neq\emptyset$. We have a commutative ladder
\[\begin{tikzcd}
	& {\Hr^i(X,\Kbf^d_\Leu)} & {\Hr^i(X,\Ibf^d_\Leu)} & {\Hr^i(X(\Rb),\Zb(\Leu))} \\
	0 & {\Hr^i(X,\bar{\Kbf}^d)} & {\Hr^i(X,\bar{\Ibf}^d)} & {\Hr^i(X(\Rb),\Zb/2)}
	\arrow[from=1-2, to=1-3]
	\arrow["\psi"', from=1-2, to=2-2]
	\arrow["{\gamma_d^i}", from=1-3, to=1-4]
	\arrow["{\pi_{X,\Leu}^{i,d}}"', from=1-3, to=2-3]
	\arrow["{\rho(\Leu)}"', from=1-4, to=2-4]
	\arrow[from=2-1, to=2-2]
	\arrow[from=2-2, to=2-3]
	\arrow["{\bar{\gamma}_d^i}"', from=2-3, to=2-4]
\end{tikzcd}\]
with exact rows where $\rho(\Leu)$ is induced by the epimorphism $\Zb(\Leu)\to\Zb/2$ of sheaves. It is obtained by taking cohomology long exact sequences in (\ref{eq:commutative_ladder_signature_variety}). Only exactness at $\Hr^i(X,\bar{\Kbf}^d)$ requires justification: it follows from (\ref{eq:commutative_ladder_signature_cohomology}). The snake lemma then yields an exact sequence \[\Ker\psi\to\Ker\pi_{X,\Leu}^{i,d}\to\Ker(\Im\gamma_d^i\to\Hr^i(X(\Rb),\Zb/2)\to\Coker\psi.\] By Lemma \ref{lem:2_torsion_to_(-1)_torsion}, the morphism $\psi$ is an isomorphism hence both outer terms in this exact sequence vanish. It follows that $\gamma_d^i$ induces an isomorphism 
\begin{equation}\label{eq:cycle_class_iso_kernels}
\Ker\pi_{X,\Leu}^{i,d}\xrightarrow{\sim}\Ker\left(\Im\gamma_d^i\hookrightarrow\Hr^i(X(\Rb),\Zb(\Leu))\to\Hr^i(X(\Rb),\Zb/2)\right).
\end{equation}
Now let $\alpha\in\Hr^i(X,\Ibf^d_\Leu)$ be $2$-torsion and suppose that $\pi_{X,\Leu}^{i,d}(\alpha)=0$. We have to show that $\alpha=0$. We observe that \[\rho(\Leu)\circ\gamma_d^i(\alpha)=\bar{\gamma}_d^i\circ\pi_{X,\Leu}^{i,d}(\alpha)=0\] so in view of the exact sequence \[\Hr^i(X(\Rb),\Zb(\Leu))\xrightarrow{\cdot 2}\Hr^i(X(\Rb),\Zb(\Leu))\xrightarrow{\rho(\Leu)}\Hr^i(X(\Rb),\Zb/2)\] associated with the epimorphism $\Zb(\Leu)\to\Zb/2$ of sheaves, there exists a cohomology class $\beta\in\Hr^i(X(\Rb),\Zb(\Leu))$ such that $\gamma_d^i(\alpha)=2\beta$. Since $\alpha$ is $2$-torsion, the class $\beta$ is $4$-torsion, hence $2$-torsion by our assumption on the torsion subgroups of $\Hr^i(X(\Rb),\Zb(\Leu))$. Thus $\gamma_d^i(\alpha)=2\beta=0$: we then have $\alpha=0$ by (\ref{eq:cycle_class_iso_kernels}), as required.
\end{proof}

\begin{cor}\label{cor:kernel_connecting_steenrod}
Under the assumption of Lemma \ref{lem:projection_dimension_filtration_isomorphism}, there is an equality \[\Ker\partial_{X,\Leu}^{i-1,d-1}=\Ker\Phi_{X,\Leu}^{i-1,d-1}\] of subgroups of $\Hr^{i-1}(X,\bar{\Ibf}^{d-1})$. Consequently, the epimorphism $\Ibf^{d-1}(\Leu)\to\bar{\Ibf}^{d-1}$ induces an exact sequence \[\Hr^{i-1}(X,\Ibf^d_\Leu)\to\Hr^{i-1}(X,\Ibf^{d-1}_\Leu)\to\Hr^{i-1}(X,\bar{\Ibf}^{d-1})\xrightarrow{\Phi_{X,\Leu}^{i-1,d-1}}\Hr^{i}(X,\bar{\Ibf}^{d})\] of abelian groups. 
\end{cor}

\begin{proof}
The group $\Hr^{i-1}(X,\bar{\Ibf}^{d-1})$ is $2$-torsion hence so is $\Im\partial_{X,\Leu}^{i-1,d-1}$. Therefore according to Lemma \ref{lem:projection_dimension_filtration_isomorphism}, the map $\pi_{X,\Leu}^{i,d}$ is injective on $\Im\partial_{X,\Leu}^{i-1,d-1}$ so that \[\Ker\partial_{X,\Leu}^{i-1,d-1}=\Ker(\pi_{X,\Leu}^{i,d}\circ\partial_{X,\Leu}^{i-1,d-1}=\Phi_{X,\Leu}^{i-1,d-1}).\] This completes the proof of the first assertion. The second follows from the first and the exact sequence \[\Hr^{i-1}(X,\Ibf^d_\Leu)\to\Hr^{i-1}(X,\Ibf^{d-1}_\Leu)\to\Hr^{i-1}(X,\bar{\Ibf}^{d-1})\xrightarrow{\partial_{X,\Leu}^{i-1,d-1}}\Hr^{i}(X,\Ibf^d_\Leu)\] provided by the definition of $\partial_{X,\Leu}^{i-1,d-1}$.
\end{proof}

We now verify the assumptions of Lemma \ref{lem:projection_dimension_filtration_isomorphism} in certain cases.

\begin{lem}\label{lem:4-torsion=2-torsion}
Let $T$ be a paracompact Hausdorff topological space and let $L$ be a line bundle on $T$. Then $\Hr^1(T,\Zb(L))[4]=\Hr^1(T,\Zb(L))[2]$.
\end{lem}

\begin{proof}
Working componentwise, we may assume that $T$ is connected. The exact sequence \[0\to\Zb(L)\xrightarrow{\cdot 4}\Zb(L)\to\Zb/4(L)\to 0\] of sheaves on $T$ induces an exact sequence \[\Hr^0(T,\Zb(L))\to\Hr^0(T,\Zb/4(L))\to\Hr^1(T,\Zb(L))\xrightarrow{\cdot 4}\Hr^1(T,\Zb(L))\] of abelian groups. If $L$ is trivial, then the map $\Hr^0(T,\Zb(L))\to\Hr^0(T,\Zb/4(L))$ is the reduction mod~$4$ map $\Zb\to\Zb/4$ (because $T$ is connected) and is therefore surjective so $\Hr^1(T,\Zb(L))$ has no $4$-torsion (in fact, it is torsion free) and the claim is clear. Therefore we may assume that $L$ is nontrivial. By the above exact sequence, to prove that $\Hr^1(T,\Zb(L))[4]$ is $2$-torsion, it suffices to prove that $\Hr^0(T,\Zb/4(L))$ is $2$-torsion. To this end, we consider the exact sequence \[0\to\Zb/2\xrightarrow{\cdot 2}\Zb/4(L)\to\Zb/2\to 0\] of sheaves on $T$. It induces an exact sequence \[0\to\Hr^0(T,\Zb/2)\to\Hr^0(T,\Zb/4(L))\to\Hr^0(T,\Zb/2)\xrightarrow{\partial}\Hr^1(T,\Zb/2)\] of abelian groups. But by \cite{greenblattHomologyLocalCoefficients2006}, the connecting homomorphism $\partial$ takes the nonzero element of the group $\Hr^0(T,\Zb/2)=\Zb/2$ to the first Stiefel--Whitney class $w_1(L)$ of $L$, which is nonzero since $L$ is nontrivial. Consequently, the morphism~$\partial$ is injective and thus the map $\Hr^0(T,\Zb/2)\to\Hr^0(T,\Zb/4(L))$ is surjective which completes the proof.
\end{proof}

\begin{rema}
The above argument shows more precisely that if $T$ is connected and $L$ is nontrivial, then the $2$-torsion section $\bar{2}$ of $\Hr^0(T,\Zb/4)$ induces a section of $\Hr^0(T,\Zb/4(L))$ which generates this $2$-torsion group.
\end{rema}

\begin{rema}\label{rema:top_twisted_cohomology}
Let $M$ be a smooth manifold of dimension $d$ and let $L$ be a line bundle on $M$. Denote by $T$ the set of compact connected components of $M$, and by~$T'$ the set of such components $V$ such that $L_{|V}$ is \emph{not} isomorphic to the orientation sheaf of $V$. Then twisted Poincaré duality yields an isomorphism \[\Hr^d(M,\Zb(L))\cong\Zb^{T\setminus T'}\oplus(\Zb/2)^{T'}\] (see \cite[Lemma 1.2.2]{faselVasersteinSymbolReal2018} for details and further references). In particular, its $4$- and $2$-torsion subgroups agree.
\end{rema}

\begin{rema}\label{rema:cases_mod_2_cycle_class_surjective}
We note here for easier future reference that Lemma \ref{lem:projection_dimension_filtration_isomorphism} applies in cohomological degree $i=d$ by Remark \ref{rema:top_twisted_cohomology}, in cohomological degree $i=1$ by Lemma \ref{lem:4-torsion=2-torsion} and in cohomological degree $i=0$ as the abelian group $\Hr^0(T,\mathbb{L})$ is free if $\mathbb{L}$ is a local system on the topological space $T$, hence in all cohomological degrees if the variety under consideration is a surface.
\end{rema}

\section{The top shifted Witt group in low dimension}\label{section:top_witt}

\begin{prop}\label{prop:small_d_w_d}
Let $X$ be a smooth algebraic variety of dimension $d$ over $\Rb$ and let $\Leu$ be a line bundle over $X$. Let $d_\Leu:\Hr^{d-1}(X(\Rb),\Zb/2)\to\Hr^d(X(\Rb),\Zb(\Leu))$ denote the connecting morphism for the cohomology long exact sequence associated with the epimorphism $\Zb(\Leu)\to\Zb/2$, and let $\Hr_\alg^{d-1}(X(\Rb),\Zb/2)$ denote the image of the homomorphism $\bar{\gamma}^{d-1}:\Ch^{d-1}(X)\to\Hr^{d-1}(X(\Rb),\Zb/2)$. If $(*)$ holds, there exists a canonical isomorphism \[\Hr^d(X,\Wbf_\Leu)\xrightarrow{\sim}\Hr^d(X(\Rb),\Zb(\Leu))/d_\Leu\Hr_\alg^{d-1}(X(\Rb),\Zb/2)\] of abelian groups. Else, there is a canonical isomorphism \[\Hr^d(X,\Wbf_\Leu)\cong\Coker(\Sq^2_\Leu:\Ch^{d-1}(X)\to\Ch^d(X))\] of abelian groups.
\end{prop}

\begin{proof}
This is essentially what is established in \cite[Theorem 3.3]{lerbetCohomologicalClassificationVector2026}. For the convenience of the reader, we give some details. By Lemma \ref{lem:easy_higher_pardon}, the inclusion $\Ibf^{d-1}(\Leu)\subseteq\Wbf(\Leu)$ induces an isomorphism $\Hr^d(X,\Ibf^{d-1}_\Leu)\cong\Hr^d(X,\Wbf_\Leu)$. Consequently, in the proof of the above proposition, we may replace $\Hr^d(X,\Wbf_\Leu)$ by $\Hr^d(X,\Ibf^{d-1}_\Leu)$, which we do from now on.

We consider the following commutative ladder:
\[\begin{tikzcd}
	{\Hr^{d-1}(X,\bar{\Ibf}^{d-1})=\Ch^{d-1}(X)} & {\Hr^d(X,\Ibf^d_\Leu)} & {\Hr^d(X,\Ibf^{d-1}_\Leu)} \\
	{\Hr^{d-1}(X(\Rb),\Zb/2)} & {\Hr^d(X(\Rb),\Zb(\Leu))}
	\arrow["{\partial_{X,\Leu}^{d-1,d-1}}", from=1-1, to=1-2]
	\arrow["{\bar{\gamma}^{d-1}}"', from=1-1, to=2-1]
	\arrow[from=1-2, to=1-3]
	\arrow["{\gamma^d}", from=1-2, to=2-2]
	\arrow["{d_\Leu}"', from=2-1, to=2-2]
\end{tikzcd}\]
of cohomology long exact sequences (see \cite[p. 190]{lerbetImageHigherSignature2026}). The cokernel of the map $\Hr^d(X,\Ibf^d_\Leu)\to\Hr^d(X,\Ibf^{d-1}_\Leu)$ is a subgroup of $\Hr^d(X,\bar{\Ibf}^{d-1})$, which vanishes by Lemma \ref{lem:easy_higher_pardon} so $\Hr^d(X,\Ibf^{d-1}_\Leu)=\Coker\partial_{X,\Leu}^{d-1,d-1}$.

Now assume that $X$ is not proper or $X(\Rb)$ is not empty. Then the map $\gamma^d:\Hr^d(X,\Ibf^d_\Leu)\to\Hr^d(X(\Rb),\Zb(\Leu))$ is an isomorphism by \cite[Theorem 3.5]{lerbetImageHigherSignature2026}. Therefore it induces an isomorphism $\chi:\Hr^d(X,\Ibf^{d-1}_\Leu)\xrightarrow{\simeq}\Hr^d(X(\Rb),\Zb(\Leu))/d_\Leu\Im\bar{\gamma}^{d-1}$ on cokernels, where $\Im\bar{\gamma}^{d-1}=\Hr_\alg^{d-1}(X(\Rb),\Zb/2)$ by definition.

We now assume that $X$ is proper and $X(\Rb)$ is empty. We consider the map $\pi_{X,\Leu}^{d,d}:\Hr^d(X,\Ibf^d_\Leu)\to\Hr^d(X,\bar{\Ibf}^d)$ induced by the epimorphism $\Ibf^d(\Leu)\to\bar{\Ibf}^d$. By (\ref{eq:projection_isomorphism}), this morphism is an isomorphism. The composite \[\Hr^{d-1}(X,\bar{\Ibf}^{d-1})\xrightarrow{\partial_{X,\Leu}^{d-1,d-1}}\Hr^d(X,\Ibf^d_\Leu)\xrightarrow{\pi_{X,\Leu}^{d,d}}\Hr^d(X,\bar{\Ibf}^d)\] is precisely $\Sq^2_\Leu:\Ch^{d-1}(X)\to\Ch^{d}(X)$ (Remark \ref{rema:twisted_steenrod_square}). Since $\Hr^d(X,\Ibf^{d-1}_\Leu)=\Coker\partial_{X,\Leu}^{d-1,d}$, this completes the proof.
\end{proof}

\begin{rema}\label{rema:top_twisted_witt_group_poincaré_duality}
Let $X$ be a smooth real algebraic variety of dimension $d$ and let $\Leu$ be a line bundle on $X$. Denote by $T$ the set of compact connected components of $X(\Rb)$ and by $T'$ the set of compact connected components $V$ of $X(\Rb)$ such that $\Leu(\Rb)_{|V}$ is \emph{not} isomorphic to the orientation sheaf $\omega_{V}$ of $V$. Recall from Remark \ref{rema:top_twisted_cohomology} that $\Hr^d(X(\Rb),\Zb(\Leu))\cong\Zb^{T\setminus T'}\oplus(\Zb/2)^{T'}$. In particular, this group is $2$-torsion free if $\Leu_{|V}\simeq\omega_V$ for each compact component $V$ of $X(\Rb)$ so the map $d_\Leu$ is trivial in this case, yielding an isomorphism $\Wr^d(X,\Leu)\cong\Zb^T$ in this case. In general, this also implies that $\Wr^d(X,\Leu)[\frac{1}{2}]\cong\Zb[\frac{1}{2}]^{T\setminus T'}$ and $\Wr^d(X,\Leu)$ is entirely governed by topology modulo $2$-torsion.
\end{rema}

\begin{rema}
The essential input in the case where $X$ is proper and $X(\Rb)$ is empty is the vanishing of the groups $\Hr^i(X,\Ibf^{d+1}_\Leu)$ for $i\in\{d,d+1\}$ (along with the observation of Remark \ref{rema:twisted_steenrod_square}). This vanishing holds as soon as $X$ is a smooth variety and the function field of $X$ has cohomological dimension $\leq d$ (see also \cite[Theorem 4.1]{zibrowiusWittGroupsCurves2014}). This is indeed the case if the base field is $\Rb$ and $X$ is a smooth $\Rb$-variety of dimension $d$ such that $X(\Rb)=\emptyset$ by \cite[Proposition 1.2.1]{colliot-theleneRealComponentsAlgebraic1990}.
\end{rema}

As noted in the introduction, the following corollary, which is an immediate consequence of Proposition \ref{prop:small_d_w_d} and of inspection of (\ref{eq:GW}), is a generalisation of the main theorem of \cite{bargeFibresAlgebriquesSurface1987}.

\begin{cor}\label{cor:small_d_w_d}
In Proposition \ref{prop:small_d_w_d}, let us further assume that $d\leq 3$. Then there is a canonical isomorphism \[\Wr^d(X,\Leu)\cong\Hr^d(X(\Rb),\Zb(\Leu))/d_\Leu\Hr_\alg^{d-1}(X(\Rb),\Zb/2)\] if $X$ satisfies $(*)$ and a canonical isomorphism \[\Wr^d(X,\Leu)\cong\Coker(\Sq_\Leu^2:\Ch^{d-1}(X)\to\Ch^{d}(X))\] otherwise.
\end{cor}

\section{Curves}\label{section:curves}


Let $X$ be a geometrically connected smooth real algebraic curve and let $\Leu$ be a line bundle on $X$. We let $Y$ be a smooth compactification of $X$, whose genus we denote by $g$, and we denote by $r$ (resp. by $c$) the number of closed points $z$ of $Z=Y\setminus X$ such that $\kappa(z)\simeq\Rb$ (resp. such that $\kappa(z)\simeq\Cb$). We denote the number of compact connected components of $X(\Rb)$ by $t$; it is then easy to see that the number of connected components of $X(\Rb)$ is $s=r+t$. We let $m$ denote the number of connected components $V$ of $X(\Rb)$ such that $\Leu(\Rb)_{|V}$ is nontrivial. We denote by $\varepsilon$ the dimension of $\CH^1(X_\Cb)/2$ as a $\Zb/2$-vector space; thus $\varepsilon=1$ if $X$ is proper and $\varepsilon=0$ else. We also recall from (\ref{eq:commutative_ladder_signature_variety}) that there is a commutative ladder
\begin{equation}\label{eq:exact_sequence_signature_curve}
\begin{tikzcd}
	0 & {\Kbf^1(\Leu)} & {\Ibf(\Leu)} & {\iota_*\Zb(\Leu)} & 0 \\
	0 & {\bar{\Kbf}^1} & {\bar{\Ibf}} & {\iota_*\Zb/2} & 0
	\arrow[from=1-1, to=1-2]
	\arrow[from=1-2, to=1-3]
	\arrow[from=1-2, to=2-2]
	\arrow["{\sign_1}", from=1-3, to=1-4]
	\arrow[from=1-3, to=2-3]
	\arrow[from=1-4, to=1-5]
	\arrow["{\mod 2}", from=1-4, to=2-4]
	\arrow[from=2-1, to=2-2]
	\arrow[from=2-2, to=2-3]
	\arrow["{\bar{\sign}_1}", from=2-3, to=2-4]
	\arrow[from=2-4, to=2-5]
\end{tikzcd}
\end{equation}
of sheaves on $X_\Zar$ with exact rows, and similarly over $Y_\Zar$ for the bottom row.

\begin{lem}\label{lem:top_cohomology_2_torsion}
One has $\Hr^1(X,\bar{\Kbf}^1)=0$ if $X$ satisfies $(*)$.
\end{lem}

\begin{proof}
This follows from \cite[Theorem 3.2 (d)]{colliot-theleneZerocyclesCohomologyReal1996}, see \cite[Proposition 3.3]{lerbetImageHigherSignature2026}.
\end{proof}

If $\Fbf$ is a $2$-torsion sheaf on a topological space $T$, we set $h^i(T,\Fbf)$ to be the dimension of the $\Zb/2$-vector space $\Hr^i(T,\Fbf)$.

\begin{lem}\label{lem:all_cohomology_of_2_torsion}
We have $h^0(X,\Kbf^1_\Leu)=g+c$ if $X$ satisfies $(*)$ and $h^0(X,\Kbf^1_\Leu)=g+1$ otherwise.
\end{lem}

\begin{proof}
We will compute $h^0(X,\bar{\Kbf}^1)$; this suffices by Lemma \ref{lem:2_torsion_to_(-1)_torsion}. Suppose first that $Y(\Rb)$ is not empty. Inspection of the Bloch--Ogus spectral sequence (\ref{eq:bloch_ogus}) shows that $\Hr^0(Y,\bar{\Ibf})\cong\Hr^0(Y,\Hsc^1)$ is isomorphic to $\Hr_\et^1(Y)$. This étale cohomology group was computed by Cox in \cite{coxEtaleHomotopyType1979} to be isomorphic to $(\Zb/2)^{g+\sigma}$ where $\sigma$ is the number of connected components of $Y(\Rb)$. Since $Y(\Rb)\neq\emptyset$, by Lemma \ref{lem:top_cohomology_2_torsion}, the group $\Hr^1(Y,\bar{\Kbf}^1)$ vanishes and the bottom row of (\ref{eq:exact_sequence_signature_curve}) induces an exact sequence \[0\to\Hr^0(Y,\bar{\Kbf}^1)\to\Hr^0(Y,\bar{\Ibf})\to\Hr^0(Y,\bar{\Ibf}^2)=\Hr^0(Y(\Rb),\Zb/2)\to\Hr^1(Y,\bar{\Kbf}^1)=0\] where $h^0(Y(\Rb),\Zb/2)=\sigma$. This yields $h^0(Y,\bar{\Kbf}^1)=g$. There is a localisation exact sequence relating the cohomology of $Y$, of $Y$ with support in $Z$ and of $X$ with coefficients in the sheaf $\bar{\Ibf}$. Since $Z$ is smooth (as a disjoint union of spectra of fields) of codimension $1$, purity identifies the cohomology of $Y$ with support in $Z$ in coefficients in $\bar{\Ibf}$ and the cohomology of $Z$ with coefficients in $\bar{\Wbf}$ (\emph{e.g.} \cite[§2.2]{faselLecturesChowWittGroups2020}). Since $Z$ has dimension $0$, the group $\Hr^1(Z,\bar{\Wbf})$ vanishes. Thus the localisation exact sequence reads 
\begin{equation}\label{eq:localisation_exact_sequence_compactification_curve}
0\to\Hr^0(Y,\bar{\Ibf})\to\Hr^0(X,\bar{\Ibf})\to\Hr^0(Z,\bar{\Wbf})\to\Hr^1(Y,\bar{\Ibf})\to\Hr^1(X,\bar{\Ibf})\to 0.
\end{equation}
Since $Y(\Rb)\neq\emptyset$, the subset $Y(\Rb)$ of $Y$ is Zariski-dense \cite[Proposition 1.1]{benoistHilberts17thProblem2017} so it meets the nonempty open subset $X$ of $Y$ hence $X(\Rb)$ is not empty. It then follows from \cite[Theorem 3.2 (d)]{colliot-theleneZerocyclesCohomologyReal1996} that $h^1(X,\bar{\Ibf})=h^1(X(\Rb),\Zb/2)=t$ and $h^1(Y,\bar{\Ibf})=h^1(Y(\Rb),\Zb/2)=\sigma$. On the other hand, since $\bar{\Wbf}=\Zb/2$ as sheaves on $Z_\Zar$ and $Z$ has $r+c$ connected components, one has $h^0(Z,\bar{\Wbf})=r+c$. Taking the alternating sum of the dimensions in (\ref{eq:localisation_exact_sequence_compactification_curve}) then yields $h^0(X,\bar{\Ibf})=g+c+s$ and the exact sequence \[0\to\Hr^0(X,\bar{\Kbf}^1)\to\Hr^0(X,\bar{\Ibf})\to\Hr^0(X(\Rb),\Zb/2)=(\Zb/2)^s\to 0\] induced by the bottom row of (\ref{eq:exact_sequence_signature_curve}) shows that $\Hr^0(X,\bar{\Kbf}^1)=(\Zb/2)^{g+c}$.

Now assume that $Y(\Rb)=\emptyset$. In particular, one has $r=0$ by definition. To compute $\dim\Hr_\et^1(Y)$, we use the real-complex exact sequence (\ref{eq:real_complex_exact_sequence}), which induces an exact sequence 
\begin{equation}\label{eq:exact_sequence_quad_ext}
0\to\Zb/2\cdot\omega\to\Hr_\et^1(Y)\to\Hr_\et^1(Y_\Cb)\to\Hr_\et^1(Y)\to\Hr_\et^2(Y)
\end{equation}
(Remark \ref{rema:real_complex_H1}) and yields an exact sequence $\Hr_\et^2(Y_\Cb)\to\Hr_\et^2(Y)\to\Hr_\et^3(Y)$. Since $Y$ does not satisfy $(*)$, by Lemma \ref{lem:top_étale_cohomology_proper_empty}, one has $\Hr_\et^2(Y)\simeq\Hr_\et^2(Y_\Cb)=\Zb/2$. In particular, the dimension $\delta$ of the image of the homomorphism $\Hr_\et^1(Y)\to\Hr_\et^2(Y)$ is less than or equal to $1$. Taking the alternating sum of the dimensions in (\ref{eq:exact_sequence_quad_ext}) then yields $1-x+2g-x+\delta=0$ where $1=\dim\Hr_\et^0(Y)$ and $2g=\dim\Hr_\et^1(Y_\Cb)$. Since $2x$ is even, we must have $\delta=1$ and thus $x=g+1$. Therefore $h^0(Y,\bar{\Ibf})=g+1$ so the exact sequence \[0\to\Hr^0(Y,\bar{\Kbf}^1)\to\Hr^0(Y,\bar{\Ibf})\to\Hr^0(Y(\Rb),\Zb/2)=0\] determined by the bottom row of (\ref{eq:exact_sequence_signature_curve}) yields $h^0(Y,\bar{\Kbf}^1)=g+1$. This completes the computation of $h^0(X,\bar{\Kbf}^1)$ if $X$ is proper as $X=Y$ in this case so assume from now on that $X$ is not proper. We use again the localisation exact sequence \[0\to\Hr^0(X,\bar{\Ibf})\to\Hr^0(Y,\bar{\Ibf})\to\Hr^0(Z,\bar{\Wbf})\to\Hr^1(Y,\bar{\Ibf})\to\Hr^1(X,\bar{\Ibf})\to 0.\] Since $X$ is not proper and $X(\Rb)\subseteq Y(\Rb)$ is empty, we have $\Hr^1(X,\bar{\Ibf})=0$ and $\Hr^1(Y,\bar{\Ibf})=\CH^1(Y)/2=\Zb/2$ by \cite[Theorem 1.3]{colliot-theleneZerocyclesCohomologyReal1996}. Finally, since $r=0$, the scheme $Z$ is the disjoint union of $c$ copies of $\Spec\Cb$ so $h^0(Z,\bar{\Wbf})=c$. Taking as before the alternating sum of the dimensions over $\Zb/2$, we obtain $h^0(X,\bar{\Ibf})=g+1+c-1=g+c$ as required.
\end{proof}

\begin{rema}
We have chosen to prove Lemma \ref{lem:all_cohomology_of_2_torsion} in the spirit of the rest of this paper. However, the group $\Hr^0(X,\Kbf^1_\Leu)$ being independent of $\Leu$ (because of Lemma \ref{lem:2_torsion_to_(-1)_torsion}), to obtain the previous statement, one can restrict to the case where $\Leu$ is trivial and use Monnier's work \cite{monnierWittGroupTorsion2002} or Knebusch's original computations \cite{knebuschAlgebraicCurvesReal1976a} to recover the untwisted group $\Hr^0(X,\Kbf^1)$. This was the route taken in the proof of \cite[Lemma 4.6.6]{lerbetRealAnalogueHodge2025} to compute $\Hr^0(X,\Kbf^1_\Leu)$ when $X$ is complete; we deduced the computation in the non-complete case from the localisation exact sequence.
\end{rema}

\begin{cor}\label{cor:computation_i_cohomology}
If $m$ denotes the number of connected components $V$ of $X(\Rb)$ such that $\Leu(\Rb)_{|V}$ is non-trivial, then $\Hr^0(X,\Ibf_\Leu)\simeq(\Zb/2)^{g+c}\oplus\Zb^{s-m}$ if $X$ satisfies $(*)$, and $\Hr^0(X,\Ibf_\Leu)\simeq(\Zb/2)^{g+1}$ otherwise.
\end{cor}

\begin{rema}
The group $\Hr^1(X,\Ibf_\Leu)$ was computed in \cite[Theorem 3.5, Remark 3.7]{lerbetImageHigherSignature2026}.
\end{rema}

\begin{proof}
This follows from the top row of (\ref{eq:commutative_ladder_signature_cohomology}) and Lemma \ref{lem:all_cohomology_of_2_torsion}.
\end{proof}

\begin{theo}\label{theo:twisted_W_0_curve}
Assume that $\Leu$ is not a square. Then there is an isomorphism $\Wr(X,\Leu)\simeq(\Zb/2)^{g+c}\oplus\Zb^{s-m}$ if $X$ satisfies $(*)$ and $\Wr(X,\Leu)\simeq(\Zb/2)^{g+1}$ otherwise.
\end{theo}

\begin{proof}
Recall that the edge morphism $\Wr(X,\Leu)\to\Hr^0(X,\Wbf_\Leu)$ in the Gersten--Witt spectral sequence is an isomorphism. The map $\Hr^0(X,\Ibf_\Leu)\to\Hr^0(X,\Wbf_\Leu)$ is an isomorphism by Lemma \ref{lem:baby_pardon}. The conclusion then follows from Corollary \ref{cor:computation_i_cohomology}.
\end{proof}

\begin{rema}
We could also deduce $\Wr(X)$ (the ``untwisted'' case) from Corollary \ref{cor:computation_i_cohomology}, but this was already done by Monnier \cite{monnierWittGroupTorsion2002} (and our proof would be very similar to his); see also \cite[Theorems 4.1, 4.6 and 4.7]{karoubiWittGroupReal2016} for a more topological argument relying on the Real Witt group (here Real is in the sense of Atiyah \cite{atiyahTHEORYREALITY1966}).
\end{rema}

\begin{rema}
Beware that although the map $\Hr^0(X,\Ibf_\Leu)\to\Hr^0(X,\Wbf_\Leu)$ is an isomorphism, the signature homomorphism $\gamma^0:\Wr(X,\Leu)\to\Hr^0(X(\Rb),\Zb(\Leu))$ does not coincide with $\gamma_1^0$. The reason is the commutative square:
\[\begin{tikzcd}
	{\Hr^0(X,\Ibf_\Leu)} & {\Hr^0(X,\Wbf_\Leu)} \\
	{\Hr^0(X(\Rb),\Zb(\Leu))} & {\Hr^0(X(\Rb),\Zb(\Leu))}
	\arrow["\cong", from=1-1, to=1-2]
	\arrow["{\gamma_1^0}"', from=1-1, to=2-1]
	\arrow["{\gamma^0}", from=1-2, to=2-2]
	\arrow["{\cdot 2}"', from=2-1, to=2-2]
\end{tikzcd}\]
of abelian groups.
\end{rema}

We now specialise Corollary \ref{cor:small_d_w_d} to the case of curves. Recall that if $M$ is a connected compact smooth manifold of dimension $1$, then $M$ is diffeomorphic to $\Sr^1$ so there is (up to isomorphism) a unique nontrivial line bundle $L$ on $M$. Since $M$ is orientable, the bundle $L$ is not isomorphic to the orientation sheaf $\omega_M$ of $M$ so $\Hr^1(M,\Zb(L))=\Zb/2$ as noted in Remark \ref{rema:top_twisted_witt_group_poincaré_duality}.

\begin{cor}\label{cor:W1_curve}
Let $X$ be a smooth real curve and let $\Leu$ be a line bundle on $X$. Let $T$ be the set of compact connected components of $X(\Rb)$ and let $T'$ be the set of components $V\in T$ such that $\Leu(\Rb)_{|V}$ is nontrivial or, equivalently, is not isomorphic to the orientation sheaf; denote the diagonal element $(\bar{1},\ldots,\bar{1})$ of $(\Zb/2)^{T'}$ by $\bar{1}$. If $(*)$ holds, then $\Wr^1(X,\Leu)\cong\Zb^T$ if $T'$ is empty and $\Wr^1(X,\Leu)\cong\Zb^{T\setminus T'}\oplus(\Zb/2)^{T'}/\Zb/2\cdot\bar{1}$ if $T'\neq\emptyset$. Else, one has $\Wr^1(X,\Leu)=\Zb/2$ if $\Leu$ is a square and $\Wr^1(X,\Leu)=0$ otherwise.
\end{cor}

\begin{rema}
If $X$ is a smooth real curve satisfying $(*)$, then $\overline{\gamma}^1:\Pic(X)/2\cong\Ch^1(X)\to\Hr^1(X(\Rb),\Zb/2)$ is an isomorphism by \cite[Theorem 1.3]{colliot-theleneZerocyclesCohomologyReal1996}. Thus the hypothesis that $T'$ is (non)empty in Corollary \ref{cor:W1_curve} is equivalent to the assumption that $\Leu$ is (not) a square.
\end{rema}

\begin{proof}
We use Corollary \ref{cor:small_d_w_d} with $d=1$. First assume that $X$ is proper and $X(\Rb)$ is empty. Then $\CH^1(X)/2=\Zb/2$ by \cite[Theorem 1.3 (b)]{colliot-theleneZerocyclesCohomologyReal1996} and $\Sq^2_\Leu:\Ch^0(X)=\Zb/2\cdot[X]\to\Ch^1(X)$ takes the unit $[X]$ of the graded ring $\Ch^*(X)$ to $\bar{c}_1(\Leu)$. Thus $\Sq^2_\Leu$ is the zero morphism if $\Leu$ is a square, and is an isomorphism of groups isomorphic to $\Zb/2$ otherwise. Since $\Wr^1(X,\Leu)=\Coker\Sq^2_\Leu$, this yields the lemma.

Now assume that $X$ satisfies $(*)$. In this case, one has $\CH^0(X)=\Zb\cdot[X]$ so $\Hr_\alg^0(X(\Rb),\Zb/2)$ is the fundamental class of $X(\Rb)$. According to \cite{greenblattHomologyLocalCoefficients2006}, the composite \[\Hr^0(X(\Rb),\Zb/2)\xrightarrow{d_\Leu}\Hr^1(X(\Rb),\Zb(\Leu))\xrightarrow{\rho(\Leu)}\Hr^1(X(\Rb),\Zb/2)\] of $d_\Leu$ and of reduction mod $2$ of the coefficients is the sum of the first Steenrod square $\Sq^1:\Hr^0(X(\Rb),\Zb/2)\to\Hr^1(X(\Rb),\Zb/2)$, which vanishes, and of cup-product with the first Stiefel--Whitney class $w_1(\Leu(\Rb))$ of $\Leu(\Rb)$. Therefore \[\rho(\Leu)\circ d_\Leu(\Hr_\alg^0(X(\Rb),\Zb/2))=\Zb/2\cdot w_1(\Leu(\Rb))\] Note that $d_\Leu\Hr_\alg^0(X(\Rb),\Zb/2)$ is contained in $(\Zb/2)^{T'}\subseteq\Hr^1(X(\Rb),\Zb(\Leu))$ because $\Hr_\alg^0(X(\Rb),\Zb/2)$ is $2$-torsion, and $\rho(\Leu)$ is an isomorphism from $(\Zb/2)^{T'}$ onto its image. Consequently, the group $d_\Leu(\Hr_\alg^0(X(\Rb),\Zb/2))$ is generated by the preimage of $w_1(\Leu(\Rb))$ by $\rho(\Leu)$: by definition, this preimage is the diagonal element $\bar{1}$ of $(\Zb/2)^{T'}$ as required. 
\end{proof}

\begin{exe}\label{exe:untwisted_witt_group_curve}
If $X$ is a smooth real curve satisfying $(*)$, then $\Wr^1(X)$ is the free abelian group on the compact connected components of $X(\Rb)$, as noted in \cite[(11.1)]{knebuschAlgebraicCurvesReal1976a} (although Knebusch's result is not expressed in this language).
\end{exe}

\section{Surfaces}\label{section:surfaces}

\subsection{Shifted Witt groups of surfaces}

We first study the Witt group of smooth real surfaces following \cite{sujathaWittGroupsReal1990}.

\begin{prop}\label{prop:exact_sequence_twisted_I}
Let $X$ be a smooth real surface and let $\Leu$ be a line bundle on $X$. Then there exists an exact sequence \[0\to\Hr^0(X,\Ibf^2_\Leu)\to\Hr^0(X,\Ibf_\Leu)\to\Hr^0(X,\bar{\Ibf})\xrightarrow{\cup\overline{c}_1(\Leu)}\Hr^1(X,\overline{\Ibf}^2)\] of abelian groups.
\end{prop}

\begin{proof}
By Corollary \ref{cor:kernel_connecting_steenrod} (see Remark \ref{rema:cases_mod_2_cycle_class_surjective}), there is an exact sequence \[0\to\Hr^0(X,\Ibf^2_\Leu)\to\Hr^0(X,\Ibf_\Leu)\to\Hr^0(X,\bar{\Ibf})\xrightarrow{\Phi_{X,\Leu}^{i-1,d-1}}\Hr^1(X,\bar{\Ibf}^2).\] Since $\Phi_X^{0,1}=0$ (Remark \ref{rema:steenrod_trivial}), by Theorem \ref{theo:twisted_differential_pardon}, the morphism $\Phi_{X,\Leu}^{0,1}$ is given by cup-product with $\bar{c}_1(\Leu)\in\Hr^1(X,\bar{\Ibf}^1)$. This completes the proof.
\end{proof}

\begin{rema}\label{rema:computation_twisted_I_étale_cohomology}
The computation of the image of $\pi_{X,\Leu}^{0,1}:\Hr^0(X,\Ibf_\Leu)\to\Hr^0(X,\bar{\Ibf})$ in Proposition \ref{prop:exact_sequence_twisted_I} can also be performed in étale cohomology, which is much more amenable to calculation. In detail, let us keep the notation of this proposition. We first use the isomorphisms $\Hr^p(X,\bar{\Ibf}^q)\cong\Hr^p(X,\Hsc^q)$ provided by the affirmation of the Milnor conjecture. Then, for $p\geq 0$, let $\Fr^1\Hr_\et^p(X)$ denote the subgroup of $\Hr_\et^p(X)$ of classes having coniveau $\geq 1$, that is, that are trivial in restriction to the complement of a closed subset of codimension $\geq 1$. Then as $\Leu$ is locally trivial, its étale Chern class $c_1^\et(\Leu)$ lies in $\Fr^1\Hr_\et^2(X)$ so $\cup c_1^\et(X):\Hr_\et^1(X)\to\Hr_\et^3(X)$ takes values in $\Fr^1\Hr_\et^3(X)$. The Bloch--Ogus spectral sequence (\ref{eq:bloch_ogus}) yields edge isomorphisms \[\Hr_\et^1(X)\xrightarrow{\sim}\Hr^0(X,\Hsc^1),\;\;\Fr_\et^1\Hr_\et^3(X)\xrightarrow{\sim}\Hr^1(X,\Hsc^2)\] such that the square
\[\begin{tikzcd}
	{\Hr_\et^1(X)} & {\Fr^1\Hr_\et^3(X)} \\
	{\Hr^0(X,\Hsc^1)} & {\Hr^1(X,\Hsc^2)}
	\arrow["{\cup c_\et^1(\Leu)}", from=1-1, to=1-2]
	\arrow["\wr"', from=1-1, to=2-1]
	\arrow["\wr", from=1-2, to=2-2]
	\arrow["{\cup\bar{c}_1(\Leu)}"', from=2-1, to=2-2]
\end{tikzcd}\]
is commutative. In particular, this yields an exact sequence \[0\to\Hr^0(X,\Ibf^2_\Leu)\to\Hr^0(X,\Ibf_\Leu)\to\Hr_\et^1(X)\xrightarrow{\cup c_1^\et(\Leu)}\Hr_\et^3(X)\] of abelian groups.
\end{rema}

\begin{exe}\label{exe:connecting_homomorphism_nontrivial}
The following example shows that there exists a smooth real surface~$X$, a line bundle $\Leu$ on $X$ and a class $\alpha\in\Hr^0(X,\overline{\Ibf})$ such that $\partial_{X,\Leu}^{0,1}(\alpha)\neq 0$, but $\gamma_2^1(\partial_{X,\Leu}^{0,1}(\alpha))=0$. Let $E$ be an elliptic curve over~$\Rb$; in particular, its real locus $E(\Rb)$ is nonempty, hence $E$ satisfies $(*)$ and is geometrically connected. Then there is an exact sequence \[0\to\Hr^0(E,\bar{\Kbf}^1)\to\Hr^0(E,\bar{\Ibf})\to\Hr^0(E(\Rb),\Zb/2)\] of abelian groups and $\Hr^0(E,\bar{\Kbf}^1)=\Zb/2$ by Lemma \ref{lem:all_cohomology_of_2_torsion} (with the notation of this lemma, we have $c=0$ and $g=1$). We let $\alpha$ be the nonzero element of $\Hr^0(E,\bar{\Kbf}^1)$. Recall that $\Hr^0(E,\bar{\Ibf})$ is naturally isomorphic to $\Hr_\et^1(E)$ (Remark \ref{rema:computation_twisted_I_étale_cohomology}) so we may view $\alpha$ as an element of $\Hr_\et^1(E)$. The real-complex exact sequence relating the mod~$2$ étale cohomology of $E$ and $E_\Cb$ reads \[0\to\Zb/2\cdot\omega\to\Hr_\et^1(E)\to\Hr_\et^1(E_\Cb)\] (Remark \ref{rema:real_complex_H1}). Since $\bar{\gamma}_1^0(\alpha)=0$ by choice, the map $\bar{\gamma}_1^0(\omega):E(\Rb)\to\Zb/2$ is constant at $1$ as $\omega=(-1)$ corresponds to $\llangle -1\rrangle$ under the isomorphism $\Hr_\et^1(E)\cong\Hr^0(E,\bar{\Ibf})$, and $E(\Rb)$ is nonempty, we have $\bar{\gamma}_0^1(\alpha)\neq\bar{\gamma}_0^1(\omega)$, hence $\alpha\neq\omega$ in $\Hr_\et^1(E)$. Consequently, the image $\alpha_{|\Cb}$ of $\alpha$ in $\Hr_\et^1(E_\Cb)$ is nontrivial.

Let again $E'$ be an elliptic curve over $\Rb$. We write $E'(\Cb)=\Cb/\Lambda'$ where $\Lambda'$ is a lattice; a basis of $\Lambda'$ induces a dual basis $(x_2,y_2)$ of $\Hr^1(E'(\Cb),\Zb)$. By \cite[Theorem 2.3]{milneAbelianVarieties2008}, the cohomology ring $\Hr^*(E'(\Cb),\Zb)$ is then the exterior algebra on $\Hr^1(E'(\Cb),\Zb)$; in particular, the group $\Hr^2(E'(\Cb),\Zb)=\Zb$ is generated by $x_2\wedge y_2$, which is therefore the class of a point of $E'$ (up to a sign). In particular, it is defined over $\Rb$. Consequently, there exists a line bundle $\Leu$ on $E'$ such that the étale Chern class $c_1^\et(\Leu_\Cb)$ of the line bundle $\Leu_\Cb$ pulled back from $\Leu$ to $E'_\Cb$ is the nontrivial element of $\Hr_\et^2(E'_\Cb,\Zb/2)\cong\Hr^2(E'(\Cb),\Zb/2)$, namely the reduction mod $2$ of $x_2\wedge y_2$. We set $\zeta=c_1^\et(\Leu)$.

Now let $A$ be the abelian surface $A=E\times E'$. We still denote by $x_2$ and $y_2$ the pullback of these classes in $\Hr^1(A(\Cb),\Zb)$ and by $\bar{x_2}$ and $\bar{y_2}$ their image in $\Hr^1(A(\Cb),\Zb/2)\cong\Hr_\et^1(A_\Cb)$. We denote by $\alpha$ and $\zeta$ the pullback to $\Hr_\et^1(A)$ of the classes previously constructed in $\Hr_\et^1(E)$ and $\Hr_\et^1(E')$ respectively, and by $\Leu$ the pullback of the bundle previously constructed on $E'$ to $A$.

We claim that $\alpha\zeta\in\Hr_\et^3(A)$ is nonzero. Indeed, write $E(\Cb)=\Cb/\Lambda$ where $\Lambda$ is a lattice; there is an induced basis $(x_1,y_1)$ of $\Hr^1(E(\Cb),\Zb)$ as explained previously and again, we still use the same notation for the pullback of these classes in $\Hr^1(A(\Cb),\Zb)$ and denote their image in $\Hr^1(A(\Cb),\Zb/2)$ by $\overline{x_1}$ and $\overline{y_1}$. Since $\alpha$ has nontrivial image in $\Hr^1(E(\Cb),\Zb/2)$, its image is of the form $\alpha_{|\Cb}=a\bar{x_1}+b\bar{y_1}$ where $a$ and $b$ are elements of $\Zb/2$ not both equal to~$0$. The graded ring $\Hr^*(A(\Cb),\Zb/2)$ is again an exterior $\Zb/2$-algebra on $\Hr^1(A(\Cb),\Zb/2)$, which itself has $(\bar{x_1},\bar{y_1},\bar{x_2},\bar{y_2})$ as a basis. In particular, it is then easy to check that $\alpha_{|\Cb}(\bar{x_2}\wedge\bar{y_2})$ is nontrivial in $\Hr^3(A(\Cb),\Zb/2)$. Since $\zeta$ pulls back to $\overline{x_2}\wedge\overline{y_2}$ in $\Hr_\et^2(A_\Cb)$, the image of $\alpha\zeta$ under the extension of scalars homomorphism $\Hr_\et^3(A)\to\Hr_\et^3(A_\Cb)$ is then nonzero: \emph{a fortiori}, the class $\alpha\zeta$ itself is nonzero in the group $\Hr_\et^3(A)$.

Now the class $\alpha$ constructed above defines an element of $\Hr^0(A,\bar{\Ibf}^1)$ whose cup-product with $\bar{c}_1(\Leu)$ is a nonzero element of $\Hr^1(A,\bar{\Ibf}^2)$ (as $\alpha c_1^\et(\Leu)$ is nonzero in $\Hr_\et^3(A)$). In particular, the class $\beta=\partial_{X,\Leu}^{0,1}(\alpha)$ is a nonzero element of $\Hr^1(A,\Ibf^2_\Leu)$. But there is a commutative square:
\[\begin{tikzcd}
	{\Hr^0(A,\overline{\Ibf})} & {\Hr^1(A,\Ibf_\Leu^2)} \\
	{\Hr^0(A,\Zb/2)} & {\Hr^1(A(\Rb),\Zb(\Leu))}
	\arrow["{\partial_{X,\Leu}^{0,1}}", from=1-1, to=1-2]
	\arrow["{ \overline{\gamma}_1^0}"', from=1-1, to=2-1]
	\arrow["{\gamma_1^2}", from=1-2, to=2-2]
	\arrow["{d_\Leu}"', from=2-1, to=2-2]
\end{tikzcd}\]
induced by the ladder of \cite[p. 190]{lerbetImageHigherSignature2026}. It follows from this square that $\gamma_2^1(\beta)=d_\Leu(\overline{\gamma}_1^0(\alpha))$ vanishes since $\overline{\gamma}_1^0(\alpha)=0$.
\end{exe}

\begin{rema}
There are easier examples of pairs $(X,\Leu)$ where $X$ is a smooth surface over $\Rb$ and $\Leu$ is a line bundle on $X$ such that $\partial_{X,\Leu}^{0,1}$ is nonzero. It suffices to take a smooth curve $C$ with nonempty real locus and a nontrivial line bundle $\Leu$ on $C$ and pullback to $X=\Pb_\Rb^1\times C$. The homomorphism $\partial_{X,\Leu}^{0,1}$ is then nontrivial because the map $\Hr^0(C,\bar{\Ibf})\to\Hr^1(C,\Ibf^2_\Leu)$ is nontrivial; this is perhaps easiest to see topologically, and comes from the fact that by choice of $\Leu$, the group $\Hr^1(C(\Rb),\Zb(\Leu))$ has nontrivial $2$-torsion. However, since the map $\gamma_2^1:\Hr^1(C,\Ibf^2_\Leu)\to\Hr^1(C(\Rb),\Zb(\Leu))$ is an isomorphism (by Jacobson's theorem, because $C$ is a curve), the image of $\partial_{X,\Leu}^{0,1}$ has zero intersection with $\Hr^1(X,\Kbf^2_\Leu)$ in this case.
\end{rema}

In the sequel, we set $\Hr_t^0(X,\Ibf^n_\Leu)=\Ker\gamma_n^0$ (resp. $\Hr_t^0(X,\Hsc^n)=\Ker\bar{\gamma}_n^0$).

\begin{lem}\label{lem:torsion_subgroup_is_torsion}
Let $X$ be a smooth $\Rb$-variety, let $\Leu$ be a line bundle on $X$ and let $n\geq 0$. Then $\Hr_t^0(X,\Ibf^n_\Leu)$ is the torsion subgroup of $\Hr^0(X,\Ibf^n_\Leu)$. Moreover, we have $\Hr_t^0(X,\Ibf^n_\Leu)=\Hr_t^0(X,\Wbf_\Leu)\cap\Hr^0(X,\Ibf^n_\Leu)$.
\end{lem}

\begin{proof}
If $\alpha\in\Hr^0(X,\Ibf^n_\Leu)$ is torsion, then $\gamma_n^0(\alpha)$ is a torsion element of the free abelian group $\Hr^0(X(\Rb),\Zb(\Leu))$ so $\gamma_n^0(\alpha)=0$. Moreover, Pfister's local-global principle \cite[Theorem 31.25]{elmanAlgebraicGeometricTheory2008} implies that $\Hr_t^0(X,\Wbf_\Leu)$ is the torsion subgroup of $\Hr^0(X,\Wbf_\Leu)$. If $\alpha\in\Hr_t^0(X,\Ibf^n_\Leu)$, then $\gamma^0(\alpha)=2^n\gamma_n^0(\alpha)=0$ so $\alpha$ is in the torsion subgroup of $\Hr^0(X,\Wbf_\Leu)$; it particular, it lies in the torsion subgroup of $\Hr^0(X,\Ibf^n_\Leu)$ and we have $\alpha\in\Hr_t^0(X,\Wbf_\Leu)\cap\Hr^0(X,\Ibf^n_\Leu)$. Conversely, if $\alpha\in\Hr_t^0(X,\Wbf_\Leu)\cap\Hr^0(X,\Ibf^n_\Leu)$, then $\gamma_n^0(\alpha)$ is well-defined and $0\gamma^0(\alpha)=2^n\gamma_n^0(\alpha)$ in the torsion-free group $\Hr^0(X(\Rb),\Zb(\Leu))$, so that $\gamma_n^0(\alpha)=0$ and $\alpha$ lies in $\Hr_t^0(X,\Ibf^n_\Leu)$ as required.
\end{proof}

\begin{cor}\label{cor:torsion_I_surface}
Let $X$ be a smooth real surface and let $\Leu$ be a line bundle on $X$ with mod $2$ first Chern class $\bar{c}_1(\Leu)$. There is an exact sequence \[0\to\Hr^0(X,\bar{\Kbf}^2)\to\Hr_t^0(X,\Ibf_\Leu)\to\Hr_t^0(X,\Hsc^1)\cap\Ker(\cup\bar{c}_1(\Leu))\to 0\] of abelian groups.
\end{cor}

\begin{proof}
We consider the following commutative ladder:
\[\begin{tikzcd}
	0 & {\Hr^0(X,\Ibf^2_\Leu)} & {\Hr^0(X,\Ibf_\Leu)} & {\Ker(\cup\bar{c}_1(\Leu))} & 0 \\
	0 & {\Hr^0(X(\Rb),\Zb(\Leu))} & {\Hr^0(X(\Rb),\Zb(\Leu))} & {\Hr^0(X(\Rb),\Zb/2)} & 
	\arrow[from=1-1, to=1-2]
	\arrow[from=1-2, to=1-3]
	\arrow["{\gamma_2^0}"', from=1-2, to=2-2]
	\arrow[from=1-3, to=1-4]
	\arrow["{\gamma_1^0}"', from=1-3, to=2-3]
	\arrow[from=1-4, to=1-5]
	\arrow["{\bar{\gamma}_1^0}"', from=1-4, to=2-4]
	\arrow[from=2-1, to=2-2]
	\arrow[from=2-2, to=2-3]
	\arrow[from=2-3, to=2-4]
\end{tikzcd}\]
determined by the quadratic real cycle class maps. Its rows are exact, in the case of the top row by Proposition \ref{prop:exact_sequence_twisted_I}. By definition of the torsion subgroups introduced above and according to the snake lemma, it induces an exact sequence \[0\to\Hr_t^0(X,\Ibf^2_\Leu)\to\Hr_t^0(X,\Ibf_\Leu)\to\Ker(\cup\bar{c}_1(\Leu))\cap\Hr_t^0(X,\Hsc^1)\to\Coker\gamma_2^0.\] But $\gamma_2^0$ is surjective by \cite[Proposition 5.2]{lerbetImageHigherSignature2026}. Moreover, one has $\Hr_t^0(X,\Ibf^2_\Leu)=\Hr^0(X,\Kbf^2_\Leu)$ by (\ref{eq:commutative_ladder_signature_variety}), and $\Hr^0(X,\Kbf^2_\Leu)\cong\Hr^0(X,\bar{\Kbf}^2)$ by Lemma \ref{lem:2_torsion_to_(-1)_torsion}. This gives the exact sequence of the statement.
\end{proof}

The following lemma was already noticed by Sujatha in \cite{sujathaWittGroupsReal1990}, see also \cite[Corollary B1]{sujathaLevelWittGroups2000}.

\begin{lem}\label{lem:global_section_top_torsion_surface}
Let $X$ be any smooth surface over $\Rb$, let $s$ be the number of connected components of $X(\Rb)$ and let $k$ be the dimension over $\Zb/2$ of the $2$-torsion subgroup of the Brauer group $\Bra(X)$ of $X$. Then $h^0(X,\bar{\Kbf}^2)=k-s$.
\end{lem}

\begin{proof}
The group $\Hr^0(X,\Hsc^2)$ can be identified with $\Bra(X)[2]$ by means of the Bloch--Ogus spectral sequence and fits in an exact sequence \[0\to\Hr^0(X,\bar{\Kbf}^2)\to\Hr^0(X,\Hsc^2)\to\Hr^0(X(\Rb),\Zb/2)\to 0\] where exactness at $\Hr^0(X(\Rb),\Zb/2)$ follows from (\ref{eq:commutative_ladder_signature_cohomology}). Thus one has $h^0(X,\bar{\Kbf}^2)=h^0(X,\Hsc^2)-h^0(X(\Rb),\Zb/2)=k-s$.
\end{proof}

If $X$ is a smooth surface over $\Rb$, we let $N(X)$ be the kernel of the homomorphism \[\omega\cup:\Hr^0(X,\Hsc^1)\to\Hr^0(X,\Hsc^2).\] 

\begin{lem}
Let $X$ be a smooth surface over $\Rb$. Then there is a short exact sequence
\begin{equation}\label{eq:exact_sequence_N}
0\to\Hr_\et^1(X_\Cb)/\pi^*\Hr_\et^1(X)\to N(X)\to\Pic'(X)/2\to 0
\end{equation}
where $\Pic'(X)/2$ is the kernel of the extension of scalars morphism $\Pic(X)/2\to\Pic(X_\Cb)/2$.
\end{lem}

\begin{proof}
This is essentially \cite[Lemma 3.1]{sujathaWittGroupsReal1990} (in that lemma, the quotient of $\Pic(X_\Cb)[2]$ by $\Pic(X)[2]$ appears instead of $\Hr_\et^1(X_\Cb)/\pi^*\Hr_\et^1(X)$: this is because Sujatha assumes that $X$ is projective so $\Hr_\et^1(X_\Cb)=\Pic(X_\Cb)[2]$ and $\Hr_\et^1(X)\simeq\Pic(X)[2]\oplus\Zb/2\cdot\omega$ where $\pi^*(\omega)=0$ and $\pi^*:\Pic(X)[2]\to\Pic(X_\Cb)[2]$ is injective). For the convenience of the reader, we recall the proof. The commutative diagram %
\[
\small{\begin{tikzcd}
	&&& {\Ch^1(X)} & {\Ch^1(X_\Cb)} \\
	{\Hr_\et^1(X)} & {\Hr_\et^1(X_\Cb)} & {\Hr_\et^1(X)} & {\Hr_\et^2(X)} & {\Hr_\et^2(X_\Cb)} \\
	0 & N & {\Hr^0(X,\Hsc^1)} & {\Hr^0(X,\Hsc^2)}
	\arrow["{\pi^*}", from=1-4, to=1-5]
	\arrow["\gamma_\et^1"', hook, from=1-4, to=2-4]
	\arrow["\gamma_\et^1", hook, from=1-5, to=2-5]
	\arrow["{\pi^*}", from=2-1, to=2-2]
	\arrow["{\pi_*}", from=2-2, to=2-3]
	\arrow["{\cup\omega}", from=2-3, to=2-4]
	\arrow["\sim"', from=2-3, to=3-3]
	\arrow["{\pi^*}", from=2-4, to=2-5]
	\arrow[from=2-4, to=3-4]
	\arrow[from=3-1, to=3-2]
	\arrow[from=3-2, to=3-3]
	\arrow["{\cup\omega}"', from=3-3, to=3-4]
\end{tikzcd}}\]
with exact rows and columns shows that modulo the isomorphism between $\Hr_\et^1(X)$ and $\Hr^0(X,\Hsc^1)$ provided by (\ref{eq:bloch_ogus}), one has $N(X)=\{\alpha\in\Hr_\et^1(X),\omega\alpha\in\Im\gamma_\et^1\}$. The above diagram then immediately yields the claim.
\end{proof}

If $X$ is a smooth surface over $\Rb$ and $\Leu$ is a line bundle over $X$, we denote by $j(\Leu)$, resp. by $l(\Leu)$, the dimension of $\Hr_t^0(X,\Hsc^1)\cap\Ker\cup\bar{c}_1(\Leu)$, resp. of $N(X,\Leu)=N(X)\cap\Ker\cup\bar{c}_1(\Leu)$. We also write $|S|$ for the cardinality of the set $S$.

\begin{prop}\label{prop:torsion_I_projective_surface}
Let $X$ be a smooth surface over $\Rb$ and let $\Leu$ be a line bundle over $X$. We then have an isomorphism $\Hr_t^0(X,\Ibf_\Leu)\simeq(\Zb/2)^{m}\oplus(\Zb/4)^{n}$ where $n=j(\Leu)-l(\Leu)$ and $m=2l(\Leu)+s-k-j(\Leu)$.
\end{prop}

\begin{proof}
The group $4\Hr_t^0(X,\Ibf_\Leu)$ is a subgroup of $\Hr_t^0(X,\Ibf^3_\Leu)$, which vanishes since $X$ is a surface (\cite[Corollary 8.11]{jacobsonRealCohomologyPowers2017}), so $4\Hr_t^0(X,\Ibf_\Leu)=0$. Hence there is an isomorphism $\Hr_t^0(X,\Ibf_\Leu)\simeq(\Zb/2)^{m}\oplus(\Zb/4)^{n}$ for a unique pair $(m,n)$ of integers. We note that $\Hr_t^0(X,\Ibf_\Leu)[2]\simeq(\Zb/2)^{m+n}$ whereas $|\Hr_t^0(X,\Ibf_\Leu)|=2^{m+2n}$ so to determine the integers $m$ and $n$, it suffices to understand $\Hr_t^0(X,\Ibf_\Leu)[2]$ and $\Hr_t^0(X,\Ibf_\Leu)$.

We claim that there is an exact sequence \[0\to\Hr_t^0(X,\Ibf^2_\Leu)[2]=\Hr_t^0(X,\Ibf^2_\Leu)\to\Ir_t(X,\Leu)[2]\to N(X,\Leu)\to 0.\] Indeed note that the first equality holds as $X$ is a surface thus the torsion in $\Hr_t^0(X,\Ibf^2_\Leu)$ is $2$-torsion. By definition, the image of $\Hr_t^0(X,\Ibf_\Leu)[2]$ in $\Hr^0(X,\Hsc^1)$ is inside $N(X,\Leu)$. Conversely, let $y\in N(X,\Leu)$. Since $y\in\Hr_t^0(X,\Hsc^1)\cap\Ker(\cup\bar{c}_1(\Leu))$, there exists $x\in\Hr_t^0(X,\Ibf_\Leu)$ with image $y$ in $\Hr^0(X,\Hsc^1)$. Then $2x$ has image $0$ in $\Hr^0(X,\Hsc^2)$ so $2x$ lifts to an element~$w$ of $\Hr^0(X,\Ibf^3_\Leu)$. Since $2x$ is torsion, so is $w$ is torsion: since $\Hr^0(X,\Ibf^3_\Leu)\simeq\Hr^0(X(\Rb),\Zb(\Leu))$ is torsion-free, we conclude that $w=0$ and thus $2x=0$. Therefore \[|\Hr_t^0(X,\Ibf_\Leu)[2]|=\frac{|N(X,\Leu)|}{|\Hr_t^0(X,\Ibf^2_\Leu)[2]|}=\frac{2^{l(\Leu)}}{2^{k-s}}=2^{l(\Leu)-k+s},\] using Lemma \ref{lem:global_section_top_torsion_surface} to compute $|\Hr_t^0(X,\Ibf^2_\Leu)[2]|=|\Hr^0(X,\bar{\Kbf}^2)|$.

Now the exact sequence \[0\to\Hr_t^0(X,\Ibf^2_\Leu)\to\Hr_t^0(X,\Ibf_\Leu)\to\Hr_t^0(X,\Hsc^1)\cap\Ker(\cup\bar{c}_1(\Leu))\to 0\] of Corollary \ref{cor:torsion_I_surface} shows that $|\Hr_t^0(X,\Ibf_\Leu)|=\frac{2^{j(\Leu)}}{2^{k-s}}$. Therefore $n=j(\Leu)-l(\Leu)$ and $m=2l(\Leu)+s-k-j(\Leu)$ as required.
\end{proof}

In the following theorem, we only consider the case where the twisting line bundle is nontrivial as the untwisted case is covered by \cite{sujathaWittGroupsReal1990,karoubiWittGroupReal2020}.

\begin{theo}\label{theo:W0_twisted_proj_surface}
Let $X$ be a smooth surface over $\Rb$ and let $\Leu$ be a line bundle on $X$ which is not a square. The group $\Wr(X,\Leu)$ fits in a split exact sequence \[0\to\Hr_t^0(X,\Ibf_\Leu)=\Hr_t^0(X,\Wbf_\Leu)\to\Wr(X,\Leu)\to\Hr^0(X(\Rb),\Zb(\Leu))\] of abelian groups, where $\Hr_t^0(X,\Ibf_\Leu)$ is described by Proposition \ref{prop:torsion_I_projective_surface}.
\end{theo}

\begin{proof}
There is an exact sequence \[0\to\Hr_t^0(X,\Wbf_\Leu)\to\Hr^0(X,\Wbf_\Leu)\to\Hr^0(X(\Rb),\Zb(\Leu))\] whose last map is the global signature map $\gamma^0$. Since $\Hr^0(X(\Rb),\Zb(\Leu))$ is a free abelian group, so is its subgroup $\Im\gamma^0$ hence the last map in the above exact sequence splits. Moreover, since $X$ is a surface, the evident map $\Wr(X,\Leu)\to\Hr^0(X,\Wbf_\Leu)$ is an isomorphism by examination of the Gersten--Witt spectral sequence. Therefore we only need to justify the equality $\Hr_t^0(X,\Ibf_\Leu)=\Hr_t^0(X,\Wbf_\Leu)$. But in fact, by Lemma \ref{lem:baby_pardon}, we have $\Hr^0(X,\Ibf_\Leu)=\Hr^0(X,\Wbf_\Leu)$, since $\Leu$ is not a square. The claimed equality now easily follows from Lemma \ref{lem:torsion_subgroup_is_torsion}.
\end{proof}

\begin{rema}
Let $X$ be a geometrically connected smooth surface over $\Rb$ such that $X(\Rb)$ is empty. As noted in \cite{sujathaWittGroupsReal1990} in the projective case, and in \cite[Lemma 1.10]{karoubiWittGroupReal2020} in general, the unit $\langle 1\rangle$ is $8$-torsion, but not $4$-torsion, in $\Wr(X)$. In fact, according to \cite[Lemma 1.10]{karoubiWittGroupReal2020}, the unit $\langle 1\rangle$ of $\Wr(X)$ is $4$-torsion if, and only if, there exists a morphism of $\Rb$-varieties from $X$ to the smooth anisotropic conic $Q_1=\{x^2+y^2+z^2=0\}\subseteq\Pb_\Rb^2$; in particular, this element of $\Wr(X)$ need not be $4$-torsion, so the $8$-torsion subgroup $\Wr(X)[8]$ is not $4$-torsion in general. In contrast, if $\Leu$ is a line bundle on $X$ which is not a square, the group $\Hr_t^0(X,\Wbf_\Leu)$ is $4$-torsion. The reason is that there is no form of rank $1$ with coefficients in a nontrivial line bundle, and such a form is precisely what would generate the element of $\Wr(X,\Leu)[8]\setminus\Wr(X,\Leu)[4]$ if it existed.
\end{rema}

Corollary \ref{cor:small_d_w_d} computes the symplectic Witt group over surfaces as follows.

\begin{prop}\label{prop:W2_smooth_surface}
Let $X$ be a smooth surface over $\Rb$ and let $\Leu$ be a line bundle on $X$. Let $d_\Leu:\Hr^1(X(\Rb),\Zb/2)\to\Hr^2(X(\Rb),\Zb(\Leu))$ denote the connecting homomorphism for the long exact sequence associated with the epimorphism $\Zb(\Leu)\to\Zb/2$. Then there is a canonical isomorphism \[\Wr^2(X,\Leu)\cong\Hr^2(X(\Rb),\Zb(\Leu))/d_\Leu\Hr_\alg^1(X(\Rb),\Zb/2)\] if $X$ satisfies $(*)$, and $\Wr^2(X,\Leu)=\Coker(\Sq^2_\Leu:\Ch^1(X)\to\Ch^2(X))$ otherwise.
\end{prop}

In the projective case, the subgroup $\Hr_\alg^1(X(\Rb),\Zb/2)$ of algebraic classes has been abundantly studied (\cite{mangolteCyclesAlgebriquesSurfaces1997,mangolteAlgebraicCyclesTopology1998}). The following two examples show that $\Wr^2(X,\Leu)$ depends on $\Leu$ in all situations where Proposition \ref{prop:W2_smooth_surface} applies.

\begin{exe}
Let $E$ be the plane cubic curve $E=\{y^2z=x^3-xz^2\}$ and set $X=E\times\mathbb{P}^1$. Since $E(\Rb)$ is nonempty, the mod $2$ cycle class map induces an isomorphism $\Ch^1(E)\cong\Hr^1(E(\Rb),\Zb/2)$ by \cite[Theorem 3.2]{colliot-theleneZerocyclesCohomologyReal1996} so there exists a line bundle $\Leu$ on $E$ such that $\Leu(\Rb)$ is nontrivial on each connected component of $E(\Rb)$. We still denote by $\Leu$ the pullback of this line bundle to $X$. As explained in \cite[Lemma 4.1.6]{asokSplittingVectorBundles2025}, the composite \[\Sq_\Leu^1:\Hr_\alg^1(X(\Rb),\Zb/2)\xrightarrow{d_\Leu}\Hr^2(X(\Rb),\Zb(\Leu))\xrightarrow{\rho(\Leu)}\Hr^2(X(\Rb),\Zb/2),\] where $\rho(\Leu)$ is the reduction mod $2$ of the coefficients, is not surjective. 
The bundle $\Leu(\Rb)$ is nontrivial on each connected component of $X(\Rb)$, whereas the orientation sheaf of $X(\Rb)$ is trivial, so $\Hr^2(X(\Rb),\Zb(\Leu))=(\Zb/2)^2$ and $\rho(\Leu)$ is an isomorphism. Proposition \ref{prop:W2_smooth_surface} then yields an isomorphism $\Wr^2(X,\Leu)=\Coker\Sq_\Leu^1\cong\Zb/2$ (generated by $x\otimes h=y\otimes h$). 
On the contrary, if $\Leu$ is trivial, then since $X(\Rb)$ is orientable, by Remark \ref{rema:top_twisted_witt_group_poincaré_duality}, we have $\Wr^2(X)\cong\Zb^2$.
\end{exe}

\begin{exe}\label{exe:W2_proper_empty_real_locus}
Let $X\subseteq\Pb_\Rb^3$ be a smooth hypersurface and let $h\in\CH^1(X)$ be the class of a hyperplane section. Suppose that $X(\Rb)=\emptyset$. This implies that the degree $\delta$ of $X$ is even; assume further that $\delta$ is congruent to $2$ mod $4$. Then as noted in \cite[Lemma 3.2.4]{asokSplittingVectorBundles2025}, one has $\overline{h}^2\neq 0$ in $\Ch^2(X)$. Recall that $\Sq^2=(x\mapsto x^2:\Ch^1(X)\to\Ch^2(X))$ by \cite[Lemma 9.8]{voevodskyReducedPowerOperations2003}. One has $\Ch^2(X)=\Zb/2$ by \cite[Theorem 1.3 (b)]{colliot-theleneZerocyclesCohomologyReal1996} so $\Sq^2$ is in fact surjective in this case. We conclude that $\Wr^2(X)=\Coker\Sq^2=0$.

On the other hand, according to the Noether--Lefschetz theorem (see \cite[Proposition 3.2.2]{asokSplittingVectorBundles2025} for the form of the statement needed here), if $X\subseteq\mathbb{P}_\Rb^3$ is a very general hypersurface of degree $\geq 4$, then $\CH^1(X)=\Zb\cdot h$ where $h\in\CH^1(X)$ is a hyperplane section. In particular, there exists a smooth surface $X$ with this property of degree $\delta=6$ such that $X(\Rb)$ is empty. Then $\bar{c}_1(\Osc_X(1))=\overline{h}\in\Ch^1(X)$ so $\Sq^2_{\overline{h}}(\overline{h})=2\overline{h}^2=0$ in the $2$-torsion group $\Ch^2(X)$. It follows that $\Sq^2_{\overline{h}}:\Ch^1(X)\to\Ch^2(X)$ is the zero map and thus $\Wr^2(X,\Osc_X(1))=\Coker\Sq_{\overline{h}}^2=\Zb/2$.
\end{exe}

We conclude this subsection with a few remarks on the first shifted Witt group. In view of Lemma \ref{lem:baby_pardon}, we begin with the group $\Hr^1(\text{--},\Ibf(\text{--}))$.

\begin{prop}\label{prop:easy_exact_sequence_W1_surface}
Let $X$ be a smooth real surface and let $\Leu$ be a line bundle on~$X$. There is an exact sequence
\[
\Hr^0(X,\bar{\Ibf}^1)\xrightarrow{\partial_{X,\Leu}^{0,1}}\Hr^1(X,\Ibf^2_\Leu)\to\Hr^1(X,\Ibf_\Leu)\to\Ker\Sq^2_\Leu\to 0
\]
of abelian groups, where the morphism $\Sq^2_\Leu:\Ch^1(X)\to\Ch^2(X)$ is given by $x\mapsto x^2+\overline{c}_1(\Leu)\cdot x$. If $\Leu$ is a square, then the homomorphism $\Hr^1(X,\Ibf^2)\to\Hr^1(X,\Ibf)$ is injective.
\end{prop}

\begin{proof}
This exact sequence is the cohomology long exact sequence for the epimorphism $\Ibf(\Leu)\to\Ibf$ of sheaves. Thus it suffices to show the following statements.
\begin{itemize}
	\item The image of $\pi_{X,\Leu}^{1,1}$ is $\Ker\Sq_\Leu^2$, where $\Sq^2_\Leu(x)=x^2+\overline{c}_1(\Leu)\cdot x$. Indeed, by Corollary \ref{cor:kernel_connecting_steenrod}, one has \[\Im\pi_{X,\Leu}^{1,1}=\Ker(\Sq_\Leu^2:\Ch^1(X)\to\Ch^2(X))\] where $\Sq_\Leu^2$ is the twisted Steenrod square of Remark \ref{rema:twisted_steenrod_square}. By \cite[Lemma 9.8]{voevodskyReducedPowerOperations2003}, the untwisted Steenrod square $\Sq^2:\Ch^1(X)\to\Ch^2(X)$ is the squaring operation $x\mapsto x^2$. This completes the proof of the first statement as $\Sq_\Leu^2=\Sq^2+\bar{c}_1(\Leu)\cdot(\text{--})$ by definition.
	\item If $\Leu$ is a square, then $\partial_X^{0,1}=0$. This was noted in Remark \ref{rema:steenrod_trivial}.
\end{itemize}
The proof is therefore complete.
\end{proof}

We can be more precise about the group $\Hr^1(X,\Ibf^2_\Leu)$ in Proposition \ref{prop:easy_exact_sequence_W1_surface}, and in fact about the group $\Hr^{d-1}(X,\Ibf_\Leu^d)$ for $X$ smooth of dimension $d$. In the following lemma and its proof, the base field for dimension is $\Zb/2$.

\begin{lem}\label{lem:torsion_H1_I2}
Let $X$ be a smooth variety over $\Rb$ of dimension $d$. Then the $\Zb/2$-vector space $\Hr^{d-1}(X,\Hsc^d)$ has dimension $\dim\Hr_\et^{2d-1}(X)-b_*+b_{d-1}+b_d$, where $b_i=\dim\Hr^i(X(\Rb),\Zb/2)$ is the $i$-th mod $2$ Betti number of $X(\Rb)$ and $b_*=\sum_i b_i$, and one has $\dim\Hr^{d-1}(X,\bar{\Kbf}^d)=\dim\Hr_\et^{2d-1}(X)-b_*+b_d$.
\end{lem}

\begin{proof}
In the Bloch--Ogus spectral sequence (\ref{eq:bloch_ogus}), the differentials $d_2^{p,q}$ vanish for $q>d$ by \cite[Theorem 2.1]{hamelTorsionZerocyclesAbelJacobi2000} and $\Er_1^{p,q}=0$ for $p<0$ and $p>q$. It follows that $\Hr^{p}(X,\Hsc^q)=\Er_\infty^{p,q}$ for $q>d$ and for $q=d$ and $p\geq d-1$. Moreover, if $q>d$, then $\Hr^p(X,\Hsc^q)\cong\Hr^p(X(\Rb),\Zb/2)$ by the main theorem of \cite{colliot-theleneRealComponentsAlgebraic1990}. Since (\ref{eq:bloch_ogus}) converges to the mod $2$ étale cohomology of $X$ (filtered by coniveau), we then have \[\dim\Hr_\et^{2d-1}(X)=\sum_{p=0}^{d-1}\Er_\infty^{p,2d-1-p}=\dim\Hr^{d-1}(X,\Hsc^d)+\sum_{p=0}^{d-2}\dim\Hr^p(X(\Rb),\Zb/2)\] so $\dim\Hr^{d-1}(X,\Hsc^d)=\dim\Hr_\et^{2d-1}(X)-b_*+b_{d-1}+b_d$. But the bottom row of (\ref{eq:commutative_ladder_signature_cohomology}) yields the equality \[\dim\Hr^{d-1}(X,\Hsc^d)=\dim\Hr^{d-1}(X,\bar{\Kbf}^d)+b_{d-1}.\] We conclude that $\dim\Hr_\et^{2d-1}(X)-b_*+b_d=\dim\Hr^{d-1}(X,\bar{\Kbf}^d)$, as required.
\end{proof}

\begin{exe}\label{exe:torsion_H1_I2}
Let $X$ be a smooth $\Rb$-variety of dimension $d$ satisfying $(*)$. By \cite[Lemma 2.2.2, Theorem 2.3.1 (b)]{colliot-theleneZerocyclesCohomologyReal1996}, there is an isomorphism
\begin{equation}\label{eq:top_cohomology_star}
\Hr_\et^{2d}(X)\cong\bigoplus_{p=0}^d\Hr^p(X(\Rb),\Zb/2)\simeq(\Zb/2)^{b_*}.
\end{equation}
Now further assume that the pullback homomorphism $\Hr_\et^{2d-1}(X)\to\Hr_\et^{2d-1}(X_\Cb)$ is surjective (this holds, \emph{e.g.}, if the group $\Hr_\et^{2d-1}(X_\Cb)$ vanishes; this is the case if $X$ is affine or is a complete intersection). Then (\ref{eq:real_complex_exact_sequence}) yields an exact sequence \[0\to\Hr_\et^{2d-1}(X)\xrightarrow{\cup\omega}\Hr_\et^{2d}(X)\xrightarrow{\pi^*}\Hr_\et^{2d}(X_\Cb)\to 0\] in which the surjectivity of $\pi^*$ follows from \cite[Lemma 2.2.2]{colliot-theleneZerocyclesCohomologyReal1996}. Thus (\ref{eq:top_cohomology_star}) yields an isomorphism $\Hr_\et^{2d-1}(X)\simeq(\Zb/2)^{b_*-\varepsilon}$ where $\varepsilon=1$ if $X$ is proper and $0$ else. By Lemma \ref{lem:torsion_H1_I2}, one then has $\Hr^{d-1}(X,\bar{\Kbf}^d)\simeq(\Zb/2)^{t-\varepsilon}$, where $t=b_d$ is the number of compact connected components of $X(\Rb)$. For instance, if $X$ is proper and $X(\Rb)$ is connected and $\Hr_\et^{2d-1}(X_\Cb)=0$, or if $X$ is affine and $X(\Rb)$ has no compact connected component, then $\Hr^{d-1}(X,\bar{\Kbf}^d)=0$.
\end{exe}

\begin{prop}\label{prop:H1I2_surface}
Let $X$ be a smooth real variety of dimension $d$ and let $\Leu$ be a line bundle on~$X$. Set $a=\dim\Hr_\et^{2d-1}(X)-b_*+b_d$. Then the morphism $\gamma_d^{d-1}$ induces a split exact sequence \[0\to(\Zb/2)^a\to\Hr^{d-1}(X,\Ibf^d(\Leu))\xrightarrow{\gamma_d^{d-1}}\Hr^{d-1}(X(\Rb),\Zb(\Leu))\to 0\] of abelian groups.
\end{prop}

\begin{proof}
In view of Lemma \ref{lem:torsion_H1_I2}, this follows from inspection of the top row of (\ref{eq:commutative_ladder_signature_cohomology}) for $p=d-1$.
\end{proof}

Thus $\Hr^{d-1}(X,\Ibf_\Leu^d)$ is governed by the topology of $X(\Rb)$ and the mod~$2$ étale cohomology of $X$ and, in certain cases (Example \ref{exe:torsion_H1_I2}), by the topology of $X(\Rb)$.

\begin{cor}\label{cor:W1_surface}
Let $X$ be a smooth real surface and let $\Leu$ be a line bundle on $X$. There is an exact sequence
\begin{equation}\label{eq:exact_sequence_W1_surface}
0\to\Coker\partial_{X,\Leu}^{0,1}\to\Wr^1(X,\Leu)\to\frac{\Ker(\Sq_\Leu^2:\Ch^1(X)\to\Ch^2(X))}{\Zb/2\cdot\overline{c}_1(\Leu)}\to 0
\end{equation}
of abelian groups.
\end{cor}

\begin{proof}
Since $X$ has dimension $2$, the Gersten--Witt spectral sequence (\ref{eq:GW}) yields an isomorphism $\Hr^1(X,\Wbf_\Leu)\cong\Wr^1(X,\Leu)$. If $\Leu$ is a square, then Lemma \ref{lem:baby_pardon} shows that $\Hr^1(X,\Wbf)=\Hr^1(X,\Ibf)$ and (\ref{eq:exact_sequence_W1_surface}) follows from Proposition \ref{prop:easy_exact_sequence_W1_surface}. Assume that $\Leu$ is not a square. Then $e(\Leu)\in\Hr^1(X,\Ibf_\Leu^1)$ maps to $\overline{c}_1(\Leu)$ under the homomorphism $\Hr^1(X,\Ibf_\Leu^1)\to\Hr^1(X,\bar{\Ibf}^1)=\Ch^1(X)$, and $\overline{c}_1(\Leu)$ is nonzero. In particular, since the sequence of Proposition \ref{prop:easy_exact_sequence_W1_surface} is a complex, the image of the map $\Coker\partial_{X,\Leu}^{0,1}\to\Hr^1(X,\Ibf_\Leu)$ does not contain $e(\Leu)$. The exact sequence of Proposition \ref{prop:easy_exact_sequence_W1_surface} thus induces an exact sequence \[0\to\Coker\partial_{X,\Leu}^{0,1}\to\frac{\Hr^1(X,\Ibf_\Leu)}{\Zb/2\cdot e(\Leu)}\to\frac{\Ker(\Sq_\Leu^1:\Ch^1(X)\to\Ch^2(X)}{\Zb/2\cdot\overline{c}_1(\Leu)}\to 0.\] Since $\Hr^1(X,\Ibf_\Leu)/\Zb/2\cdot e(\Leu)=\Hr^1(X,\Wbf_\Leu)$ by Lemma \ref{lem:baby_pardon}, this completes the proof.
\end{proof}

\begin{rema}
When $\Leu$ is a square, so that $\overline{c}_1(\Leu)=0$, Corollary \ref{cor:W1_surface} yields an exact sequence \[0\to\Hr^1(X,\Ibf^2)\to\Wr^1(X)\to\Ker(\Sq^2:\Ch^1(X)\to\Ch^2(X))\to 0.\] This follows from the fact that $\partial_X^{0,1}$ is trivial (Remark \ref{rema:steenrod_trivial}).
\end{rema}

\begin{exe}
Let $X$ be a very general hypersurface in $\Pb_\Rb^3$ (in particular, the variety $X$ is smooth) such that $X(\Rb)=\emptyset$; denote by $i$ the closed immersion $X\hookrightarrow\Pb_\Rb^3$. By the Noether--Lefschetz theorem, the group $\CH^1(X)$ is freely generated by the class of a hyperplane section $h$ (see, \emph{e.g.}, \cite[Proposition 3.2.2]{asokSplittingVectorBundles2025}). Now assume that the degree $\delta$ of $X$ is congruent to $2$ mod $4$. Then $h^2\neq 0$ by \cite[Lemma 3.2.4]{asokSplittingVectorBundles2025} so $\Sq^2:\Ch^1(X)\to\Ch^2(X)$ is injective. We then have $\Hr_\et^3(X)=\Zb/2$ by Corollary \ref{cor:end_real_complex_proper_empty} (and the proof of Lemma \ref{lem:top_étale_cohomology_proper_empty}), thus $\Hr^1(X,\bar{\Kbf}^2)=\Zb/2$ by Lemma \ref{lem:torsion_H1_I2} since $X(\Rb)$ is empty. This also implies that $\Hr^1(X,\Ibf^2)=\Zb/2$ by Proposition \ref{prop:H1I2_surface}. It now follows from Corollary \ref{cor:W1_surface} that $\Wr^1(X)=\Zb/2$. The same argument shows that if instead $\delta$ is congruent to $0$ mod $4$, then $\Hr^1(X,\Ibf_{\Osc(1)}^2)=\Zb/2$ and thus $\Wr^1(X,\Osc(1))=0$ by Lemma \ref{lem:baby_pardon}. 
\end{exe}

\begin{rema}
The exact sequence (\ref{eq:exact_sequence_W1_surface}) need not split. For example, let $C$ be a smooth curve such that $\Wr(C)[4]\neq 0$ (this holds if $C$ is proper and geometrically connected and $C(\Rb)$ is empty by \cite[Theorem 2.9]{monnierWittGroupTorsion2002}) and set $X=C\times\Pb_\Rb^1$. Since $X(\Rb)=\emptyset$, one has $\Hr^1(X,\bar{\Kbf}^2)=\Hr^1(X,\Ibf^2)$ so $\Hr^1(X,\Ibf^2)$ is $2$-torsion, and $\Ker\Sq^2\subseteq\Ch^1(X)$ is evidently $2$-torsion, if (\ref{eq:exact_sequence_W1_surface}) split, the group $\Wr^1(X)$ would be $2$-torsion. However, the projective bundle formula for Witt groups (\cite[Theorem 9.10]{schlichtingHermitianKtheoryDerived2017}) shows that $\Wr(C)$, which is not $2$-torsion by choice, is a direct summand of $\Wr^1(X)$. Nertheless, by Proposition \ref{prop:easy_exact_sequence_W1_surface}, Corollary \ref{cor:W1_surface} and Lemma \ref{lem:4-torsion=2-torsion}, the torsion subgroup of $\Wr^1(X,\Leu)$ is $4$-torsion for any smooth real surface $X$ and any line bundle~$\Leu$ on~$X$.
\end{rema}

\begin{rema}\label{rema:no_analogue_W1_surfaces}
Let $X$ be a smooth real surface and let $\Leu$ be a line bundle on $X$. The analogue of Corollary \ref{cor:small_d_w_d} cannot hold for $\Wr^1(X,\Leu)$ since the homomorphism $\gamma_2^1:\Hr^1(X,\Ibf^2_\Leu)\to\Hr^1(X(\Rb),\Zb(\Leu))$ is not injective if $X$ is affine and $X(\Rb)$ has a compact connected component (see Example \ref{exe:torsion_H1_I2}). One could, however, hope for an exact sequence
\begin{equation}\label{eq:no_analogue_W1_surfaces}
0\to\Hr^1(X,\Kbf^2_\Leu)\oplus\frac{\Hr^1(X(\Rb),\Zb(\Leu))}{d_\Leu\Im\bar{\gamma}_1^0}\to\Hr^1(X,\Ibf_\Leu)\to\Ker\Sq^2_\Leu\to 0,
\end{equation}
in other words that the injectivity defect of $\gamma_2^1$ survives in $\Hr^1(X,\Ibf_\Leu)$ and the contribution of the connecting homomorphism $\partial_{X,\Leu}^{0,1}:\Hr^0(X,\bar{\Ibf})\to\Hr^1(X,\Ibf^2_\Leu)$ to $\Hr^1(X,\Ibf_\Leu)$ can be detected purely topologically (modulo the knowledge of the image of $\bar{\gamma}_1^0:\Hr^0(X,\bar{\Ibf})\to\Hr^0(X(\Rb),\Zb/2)$, which can certainly be accessed in special cases). Unfortunately, as we saw in Example \ref{exe:connecting_homomorphism_nontrivial}, this is not the case: if $X$ is a product of elliptic curves, the intersection $\Im\partial_{X,\Leu}^{0,1}\cap\Hr^1(X,\Kbf^2_\Leu)$ is nonzero for specific choices of $\Leu$. It is unclear to us whether one can expect a manageable description of $\Coker\partial_{X,\Leu}^{0,1}$ in general. If $\Hr^1(X,\bar{\Kbf}^2)=0$, then one can prove as we did for Proposition \ref{prop:small_d_w_d} that there is an exact sequence like (\ref{eq:no_analogue_W1_surfaces}).
\end{rema}

\subsection{An example: the smooth anisotropic quadric}

Consider the smooth quadric anisotropic surface \[X=\{x_0^2+x_1^2+x_2^2+x_3^2=0\}\subseteq\Pb_\Rb^3.\] Our aim in this subsection is to compute its Witt groups using the methods explained in the previous section.

\begin{rema}
The groups $\Wr^p(X,\Osc_X(n))$ were computed by Xie \cite{xieWittGroupsSmooth2019,xieHermitianKtheoryQuadric2026}. However, our methods, being topological, are quite different; moreover, as we will see, in the two-dimensional case, there are exceptional twists to consider, related to the fact that the Picard group of a split quadric surface has rank $2$ contrary to the higher dimensional case where it is generated by $\Osc(1)$, which we believe fit more natural within our techniques.
\end{rema}

Let us begin with the group $\CH^1(X)$. The quadric $X_\Cb$ is split so $\CH^1(X_\Cb)$ is freely generated by lines $e_1$ and $e_2$, thus the hyperplane section in $\CH^1(X_\Cb)$ is $h=e_1+e_2$. The morphism $\pi^*:\CH^1(X)\to\CH^1(X_\Cb)$ induced by the projection $\pi$ is injective so we regard $\CH^1(X)$ as a subgroup of $\CH^1(X_\Cb)$. Set $f=2e_1$. The following lemma then follows from \cite[(2.7)]{karpenkoAlgebrogeometricInvariantsQuadratic1991}.

\begin{lem}
The subgroup $\CH^1(X)$ of $\CH^1(X_\Cb)$ is generated by $(f,h)$.
\end{lem}

This leaves four twists to consider: the trivial one, the twist by $\Osc_X(1)$ and the twists by $\Leu_0\otimes\Osc_X(n)$ where $\Leu_0$ is a line bundle on $X$ with first Chern class $f$ and $n\in\{0,1\}$. We first compute the twisted Steenrod square $\Sq_\Leu^2:\Ch^1(X)\to\Ch^2(X)$. Let $x$ be a closed point of $X$; we then have $\CH^2(X)=\Zb\cdot[x]$ by \cite[(2.6)]{karpenkoAlgebrogeometricInvariantsQuadratic1991}.

\begin{lem}\label{lem:steenrod_square_computation_anisotropic_quadric}
One has $\overline{f}^2=0$ and $\overline{f}\cdot\overline{h}=[x]$ in $\Ch^2(X)$.
\end{lem}

\begin{proof}
We will compute $f^2$ and $fh$ in $\CH^2(X)$, then reduce mod $2$. Denoting by $x'$ the point of $X_\Cb$ lying over $x$, one has $\pi^*([x])=2[x']$ by definition of the pullback on Chow groups for flat maps. Since $\CH^2(X_\Cb)=\Zb\cdot[x']$ by \cite[(2.6)]{karpenkoAlgebrogeometricInvariantsQuadratic1991}, we conclude that $\pi^*:\CH^2(X)\to\CH^2(X_\Cb)$ is injective. Since $e_1^2=0$ in $\CH^2(X_\Cb)$ by \cite[Exercise 68.3 (2)]{elmanAlgebraicGeometricTheory2008}, and $\pi^*$ is a ring homomorphism, this implies that $\pi^*(f^2)=\pi^*(f)^2=4e_1^2=0$, hence $f^2=0$ in $\CH^2(X)$. Similarly, one has \[\pi^*(fh)=\pi^*(f)\pi^*(h)=2(e_1h)=2[x']\] by \cite[Proposition 68.1]{elmanAlgebraicGeometricTheory2008}, so $\pi^*(fh)=\pi^*([x])$ and thus $fh=[x]$ in $\CH^2(X)$.
\end{proof}

\begin{prop}\label{prop:steenrod_squares_quadric}
The map $\Sq_{\overline{f}}^2:\Ch^1(X)\to\Ch^2(X)$ is the zero map. If $\Leu\in\Pic(X)$ has $\overline{c}_1(\Leu)\neq f$ in $\Ch^1(X)$, then $\Sq_\Leu^2$ is surjective and its kernel is generated by $f$ if $\Leu$ is a square and by $\overline{c}_1(\Leu)$ if $\Leu$ is not a square.
\end{prop}

\begin{proof}
This easily follows from Lemma \ref{lem:steenrod_square_computation_anisotropic_quadric} as $\Sq_\alpha^2(\beta)=\beta^2+\alpha\beta$ for any $(\alpha,\beta)\in\Ch^1(X)\times\Ch^1(X)$ by definition and $\Ch^2(X)=\Zb/2\cdot[x]$.
\end{proof}

\begin{lem}\label{lem:étale_chern_class_f}
One has $\gamma_\et^1(\overline{f})=\omega^2$.
\end{lem}

\begin{proof}
Since $\pi^*(f)\in\CH^1(X_\Cb)$ is a multiple of $2$, one has $\pi^*(\overline{f})=0$, hence $\pi^*(\gamma_\et^1(\overline{f}))=0$ in $\Hr_\et^2(X_\Cb)$ by compatibility of the first étale Chern class with pullbacks. Thus by (\ref{eq:real_complex_exact_sequence}), one has $\gamma_\et^1(\overline{f})=\omega\alpha$ for some $\alpha\in\Hr_\et^1(X)$. The class $\alpha$ is nonzero since $\gamma_\et^1$ is injective by inspection of (\ref{eq:bloch_ogus}); since $\Hr_\et^1(X)=\Zb/2\cdot\omega$ by Remark \ref{rema:real_complex_H1}, we conclude that $\alpha=\omega$, hence $\gamma_\et^1(\overline{f})=\omega$.
\end{proof}

\begin{lem}\label{lem:computation_étale_cohomology_quadric}
Set $\lambda=\gamma_\et^1(\overline{h})$. One then has $\omega\lambda\neq 0$ and $\omega^3=0$ in $\Hr_\et^3(X)$.
\end{lem}

\begin{proof}
The relation $\omega^3=0$ follows from \cite[Proposition 4.5 (ii)]{benoistWuRelationsReal2026}. In view of the exact sequence \[\Hr_\et^2(X_\Cb)\xrightarrow{\pi_*}\Hr_\et^2(X)\xrightarrow{\cup\omega}\Hr_\et^3(X),\] it suffices to prove that $\lambda\notin\Im\pi_*$. By compatibility of the first étale Chern class with pushforwards, and since $c_1^\et:\Ch^1(X)\to\Hr_\et^2(X)$ is injective by inspection of (\ref{eq:bloch_ogus}) and $c_1^\et(\Osc_X(1))=\lambda$, it suffices to show that $\overline{h}\in\Ch^1(X)$ does not lie in the image of $\pi_*:\Ch^1(X_\Cb)\to\Ch^1(X)$. Suppose on the contrary that $\pi_*(\alpha)=h+2\beta$ for some $\alpha\in\CH^1(X_\Cb)$ and $\beta\in\CH^1(X)$. Then $\pi^*\pi_*\alpha=\pi^*(h)+2\pi^*(\beta)$ so $\pi^*(h)=\pi^*\pi_*\alpha-2\pi^*(\beta)$. One has $\pi^*\pi_*=\Id+\sigma^*$ where $\sigma$ is the automorphism of the $\Rb$-variety $X_\Cb$ induced by complex conjugation. By \cite[Remark 1]{colliot-thlneRealRationalSurfaces1992}, the action of $\sigma$ on $\CH^1(X_\Cb)$ is trivial so $\pi^*\pi_*\alpha=2\alpha$. We conclude that $\pi^*(h)=e_1+e_2$ is a multiple of $2$ in $\CH^1(X_\Cb)=\Zb e_1\oplus\Zb e_2$: contradiction.
\end{proof}

\begin{theo}\label{theo:witt_groups_anisotropic_quadric}
Let $\Leu$ be a line bundle on the smooth anisotropic quadric surface $X$. If $\overline{c}_1(\Leu)\in\{\overline{h},\overline{f}+\overline{h}\}$, then $\Wr^p(X,\Leu)=0$ for $p\in\{0,2\}$ and the map $\Wr^1(X,\Leu)\to\Ch^1(X)/\Zb/2\cdot\overline{c}_1(\Leu)$ in Corollary \ref{cor:W1_surface} is an isomorphism. One has \[\Wr(X)=\Zb/4\cdot\langle 1\rangle,\;\;\Wr^2(X)=0\] and $\Wr^1(X)$ fits in an exact sequence \[0\to\Zb/2\to\Wr^1(X)\to\Zb/2\cdot\overline{f}\to 0.\] Finally, if $\overline{c}_1(\Leu)=\overline{f}$, then \[\Wr(X,\Leu)=\Zb/2,\;\;\Wr^2(X,\Leu)=\Zb/2\] and there is an exact sequence \[0\to\Zb/2\to\Wr^1(X,\Leu)\to\frac{\Ch^1(X)}{\Zb/2\cdot\overline{f}}\to 0\] of abelian groups.
\end{theo}

\begin{proof}
The computation of the symplectic Witt groups $\Wr^2(X,\text{--})$ follows from the description of the twisted Steenrod square in Proposition \ref{prop:steenrod_squares_quadric}. The computation of the groups $\Wr^p(X)$ follows from \cite[Theorem 5]{kahnMotivicCohomologyUnramified2000} for $p=0$, and from Corollary \ref{cor:W1_surface} for $p=1$, noting that $\Ker\Sq^2=\Zb/2\cdot\overline{f}$ by Proposition \ref{prop:steenrod_squares_quadric}.

Suppose that $\overline{c}_1(\Leu)\in\{\overline{h},\overline{h}+\overline{f}\}$. Then $\Ker\Sq_\Leu^2=\Zb/2\cdot\overline{c}_1(\Leu)$ by Proposition \ref{prop:steenrod_squares_quadric} so by Corollary \ref{cor:W1_surface}, one has $\Wr^1(X,\Leu)=\Coker\partial_{X,\Leu}^{0,1}$. Thus to show that $\Wr^1(X,\Leu)=0$, it suffices to prove that $\partial_{X,\Leu}^{0,1}$ is surjective. But $\Hr_\et^3(X)=\Zb/2$ by Corollary \ref{cor:end_real_complex_proper_empty} as $\Hr_\et^3(X_\Cb)=0$ since $X_\Cb$ is a hypersurface. Therefore by Lemma \ref{lem:computation_étale_cohomology_quadric}, one has $\Hr_\et^3(X)=\Zb/2\cdot\omega\lambda$. In particular, since $X(\Rb)$ is empty, one has $\Hr^1(X,\Ibf_\Leu^2)=\Zb/2$ by Proposition \ref{prop:H1I2_surface}. It follows that $\partial_{X,\Leu}^{0,1}$ is surjective if, and only if, it is nonzero. Since $\Hr^0(X,\overline{\Ibf})=\Hr^0(X,\Hsc^1)=\Hr_\et^1(X)$ by inspection of (\ref{eq:bloch_ogus}) and $\Hr_\et^1(X)=\Zb/2\cdot\omega$ by Remark \ref{rema:real_complex_H1}, by Remark \ref{rema:computation_twisted_I_étale_cohomology}, the surjectivity of $\partial_{X,\Leu}^{0,1}$ is therefore equivalent to the assertion that $\omega c_1^\et(\Leu)$ is nonzero in $\Hr_\et^3(X)$. Since $c_1^\et(\Leu)\in\{\lambda,\lambda+\omega^2\}$ by Lemma \ref{lem:étale_chern_class_f}, this follows from Lemma \ref{lem:computation_étale_cohomology_quadric} and completes the description of $\Wr^1(X,\Leu)=0$. One has $\Hr^0(X,\Hsc^2)=0$ by \cite[Theorem 4.5]{kahnMotivicCohomologyUnramified2000} so $\Hr^0(X,\bar{\Kbf}^2)=0$; since $X(\Rb)=\emptyset$, it then follows from (\ref{eq:commutative_ladder_signature_cohomology}) that $\Hr^0(X,\Ibf^2_\Leu)=0$. Since $\partial_{X,\Leu}^{0,1}$ is a nonzero map with source $\Hr^0(X,\overline{\Ibf})=\Zb/2$, it is injective and thus the exact sequence \[0\to\Hr^0(X,\Ibf_\Leu^2)\to\Hr^0(X,\Ibf_\Leu)\to\Hr^0(X,\bar{\Ibf})\xrightarrow{\partial_{X,\Leu}^{0,1}}\Hr^1(X,\Ibf_\Leu^2)\] shows that $\Hr^0(X,\Ibf_\Leu)=0$. By Lemma \ref{lem:baby_pardon}, one has $\Hr^0(X,\Ibf_\Leu)=\Hr^0(X,\Wbf_\Leu)$ as $\Leu$ is not a square so $\Hr^0(X,\Wbf_\Leu)=0$ and thus $\Wr(X,\Leu)$ by inspection of (\ref{eq:GW}).

It remains to prove that $\Wr(X,\Leu)=\Zb/2$ if $\overline{c}_1(\Leu)=\overline{f}$. As in the previous paragraph, one has $\Hr^0(X,\Ibf_\Leu^2)=0$ and $\Hr^0(X,\Ibf_\Leu)=\Hr^0(X,\Wbf_\Leu)\cong\Wr(X,\Leu)$ so it suffices to prove that $\Hr^0(X,\Ibf_\Leu)=0$. Remark \ref{rema:computation_twisted_I_étale_cohomology} induces an exact sequence \[0\to\Hr^0(X,\Ibf_\Leu)\to\Hr_\et^1(X)\xrightarrow{\cup c_1^\et(\Leu)}\Hr_\et^3(X).\] Since $c_1^\et(\Leu)=\omega^2$ by Lemma \ref{lem:étale_chern_class_f}, the group $\Hr_\et^1(X)$ is generated by $\omega$ by Remark \ref{rema:real_complex_H1} and $\omega^3=0$ by Lemma \ref{lem:computation_étale_cohomology_quadric}, the map $\Hr^0(X,\Ibf_\Leu)\to\Hr_\et^1(X)=\Zb/2$ is in fact surjective, hence an isomorphism, which completes the proof.
\end{proof}

\begin{rema}
We do not know to which extension classes the exact sequences in Theorem \ref{theo:witt_groups_anisotropic_quadric} correspond, \emph{a fortiori} whether these classes agree.
\end{rema}

\begin{rema}\label{rema:connecting_homomorphism_nonzero}
As noted in the proof of Theorem \ref{theo:witt_groups_anisotropic_quadric}, one has $\partial_{X,\Leu}^{0,1}(\llangle -1\rrangle)\neq 0$ if $\overline{c}_1(\Leu)\in\{\overline{h},\overline{h}+\overline{f}\}$. This is another instance of the phenomenon of Example \ref{exe:connecting_homomorphism_nontrivial}, this time without real points. We do not know if there exists a smooth real surface $Y$ and a line bundle $\Leu$ on $Y$ such that $Y(\Rb)$ is nonempty and $\partial_{Y,\Leu}^{0,1}(\llangle -1\rrangle)$ is a nonzero element of $\Hr^1(X,\Kbf_\Leu^2)$.
\end{rema}

\subsection{The image of the global signature}\label{subsection:image_global_signature_surfaces}

In this subsection, we study the image of the global signature homomorphism \[\gamma^0:\Hr^0(X,\Wbf_\Leu)\to\Hr^0(X(\Rb),\Zb(\Leu))\] for a smooth integral surface $X$ over $\Rb$ together with a line bundle $\Leu$. The general framework of this investigation is given by the following lemma.

\begin{lem}\label{lem:image_signature_I_surface}
Let $\alpha\in\Hr^0(X(\Rb),\Zb(\Leu))$. Then $\alpha$ lies in the image of the map~$\gamma_1^0$ if, and only if, its reduction mod $2$ in $\Hr^0(X(\Rb),\Zb/2)$ lies in the image of the restriction $\bar{\gamma}_1^0:(\Ker(\cup\bar{c}_1(\Leu))\subseteq\Hr^0(X,\Hsc^1))\to\Hr^0(X(\Rb),\Zb/2)$.
\end{lem}

\begin{rema}
This lemma is analogous to \cite[Proposition 4.4]{lerbetImageHigherSignature2026} (or \cite[Corollary 5.3]{hornbostelFewComputationsReal2024}). However, contrary to the situation in these statements, it would not be reasonable to call the classes in the image of $\bar{\gamma}_1^0$ algebraic as $\Hr^0(X,\Hsc^1)=\Hr_\et^1(X)$ is not a group of (cohomology classes of) algebraic cycles.
\end{rema}

\begin{proof}
Let $\rho(\Leu):\Hr^0(X(\Rb),\Zb(\Leu))\to\Hr^0(X(\Rb),\Zb/2)$ be reduction mod $2$. We then have a commutative ladder with exact rows:
\[\begin{tikzcd}
	{\Hr^0(X,\Ibf^2_\Leu)} & {\Hr^0(X,\Ibf_\Leu)} & {\Hr^0(X,\Hsc^1)} \\
	{\Hr^0(X(\Rb),\Zb(\Leu))} & {\Hr^0(X(\Rb),\Zb(\Leu))} & {\Hr^0(X(\Rb),\Zb/2)}
	\arrow[from=1-1, to=1-2]
	\arrow["{\gamma_2^0}"', from=1-1, to=2-1]
	\arrow[from=1-2, to=1-3]
	\arrow["{\gamma_1^0}"', from=1-2, to=2-2]
	\arrow["{\bar{\gamma}_1^0}"', from=1-3, to=2-3]
	\arrow[from=2-1, to=2-2]
	\arrow["\rho(\Leu)"', from=2-2, to=2-3]
\end{tikzcd}\]
in which $\gamma_2^0$ is surjective by \cite[Corollary 5.2]{lerbetImageHigherSignature2026}. Moreover, the morphism $\Hr^0(X,\Ibf_\Leu)\to\Hr^0(X,\Hsc^1)$ has image $\Ker(\cup\bar{c}_1(\Leu))$ by Proposition \ref{prop:exact_sequence_twisted_I}. The claim then follows from the four lemma.
\end{proof}

To compute the image of $\gamma^0:\Hr^0(X,\Wbf_\Leu)\to\Hr^0(X(\Rb),\Zb(\Leu))$, one then proceeds as follows. By definition of the normalised signature maps, one has a commutative square
\[\begin{tikzcd}
	{\mathrm{H}^0(X,\mathbf{I}_\Leu)} & {\mathrm{H}^0(X,\mathbf{W}_\Leu)} \\
	{\mathrm{H}^0(X(\mathbb{R}),\mathbb{Z}(\Leu))} & {\mathrm{H}^0(X(\mathbb{R}),\mathbb{Z}(\Leu))}
	\arrow[from=1-1, to=1-2]
	\arrow["{\gamma_1^0}"', from=1-1, to=2-1]
	\arrow["{\gamma^0}", from=1-2, to=2-2]
	\arrow["{\cdot 2}"', from=2-1, to=2-2]
\end{tikzcd}\]
If $\Leu$ is not a square, then the map $\Hr^0(X,\Ibf_\Leu)\to\Hr^0(X,\Wbf_\Leu)$ is surjective (hence an isomorphism) by Lemma \ref{lem:baby_pardon}. Therefore $\Im\gamma^0=2\Im\gamma_1^0$. If $\Leu$ is a square (hence is omitted from the notation), then the exact sequence \[0\to\Hr^0(X,\Ibf)\to\Hr^0(X,\Wbf)\to\Hr^0(X,\Wbf/\Ibf)\to 0\] used in Lemma \ref{lem:baby_pardon} shows that $\Hr^0(X,\Wbf)$ is generated by $\langle 1\rangle$ and $\Hr^0(X,\Ibf)$. Hence by the above commutative square, the image of $\gamma^0$ is generated by $\gamma^0(\langle 1\rangle)$, which is the diagonal element $1$ of $\Hr^0(X(\Rb),\Zb)$ by definition, and by $2\Im\gamma_1^0$.

\begin{exe}
Let $X$ be any geometrically connected smooth real surface such that $\Hr_\et^1(X_\Cb)=0$. Then $\Hr_\et^1(X)=\Zb/2\cdot\omega$ by Remark \ref{rema:real_complex_H1} and thus \[\Im\gamma^0=\Zb\cdot 1+4\Hr^0(X(\Rb),\Zb).\] For example, one has $\Hr_\et^1(X_\Cb)=0$ if $X$ is proper and geometrically rational; the above equality then recovers \cite[Theorem 5.2 (i)]{monnierUnramifiedCohomologyQuadratic2000}.
\end{exe}

We now discuss the image of the global signature for real Enriques surfaces. Let $X$ be such a surface. Recall from \cite{degtyarevHalvesRealEnriques1996} the notion of \emph{halves} of $X$: there are two degree $2$ covers $Y_i\to X$ of $X$ by K3-surfaces $Y_1$ and $Y_2$ whose complexification is the degree $2$ cover $Y_\Cb\to X_\Cb$ of the complex Enriques surface $X_\Cb$ by a K3-surface, and the halves $X^i(\Rb)$ of $X$ are the disjoint images of the induced maps $Y_i(\Rb)\to X(\Rb)$ on real points; one then has $X(\Rb)=X^1(\Rb)\cup X^2(\Rb)$. If $W$ is an open and closed subset of $X(\Rb)$, we denote by $[W]$ the sum of the generators of $\Hr^0(X(\Rb),\Zb)$ whose union is $W$; thus $[X^i(\Rb)]$ is defined. The following proposition then recovers and extends the results of \cite[§6]{monnierUnramifiedCohomologyQuadratic2000}.

\begin{prop}
Let $X$ be a real Enriques surface and let $W$ be a half of $X$. Then $\Im\gamma^0=\Zb\cdot 1+\Zb\cdot 2[W]+4\Hr^0(X(\Rb),\Zb)$ where $1\in\Hr^0(X(\Rb),\Zb)$ is the unit.
\end{prop}

\begin{proof}
Since $X$ is proper, so that $\Osc(X)^\times/\Osc(X)^{\times 2}\cong\Zb/2\cdot\omega$, the Kummer exact sequence reads \[0\to\Zb/2\cdot\omega\to\Hr_\et^1(X)\to\Pic(X)[2]\to 0,\] where $\Pic(X)[2]=\omega_X$ as $X$ is an Enriques surface. If $\alpha:Y\to X$ is a degree $2$ cover whose class $[\alpha]$ in $\Hr_\et^1(X)$ lifts $\omega_X$ along the map $\Hr_\et^1(X)\to\Pic(X)[2]$, the group $\Hr_\et^1(X)$ is therefore generated by $[\alpha]$ and $[\alpha]+\omega$. By Lemma \ref{lem:image_signature_I_surface} and the subsequent discussion, it suffices to establish that $\bar{\gamma}_0^1([\alpha])$ is the image in $\Hr^0(X(\Rb),\Zb/2)$ of $[Z]$ for some half $[Z]$. Indeed, since $X(\Rb)$ is the disjoint union of its halves, and since $\bar{\gamma}_0^1([\alpha])+\bar{\gamma}_0^1([\alpha]+\omega)=1$, if this is the case, then either $\bar{\gamma}_0^1([\alpha])=[W]$ or $\bar{\gamma}_0^1([\alpha]+\omega)=[W]$.

Let $Z$ be the image of $\alpha(\Rb):Y(\Rb)\to X(\Rb)$ and let $Z'$ be the complement of $Z$ in $X(\Rb)$; these are the halves of $X(\Rb)$. We claim that $\overline{\gamma}_0^1(\alpha)=Z'$. To see this, recall the definition of $\bar{\gamma}_1^0$ from \cite{colliot-theleneRealComponentsAlgebraic1990}. This map takes a class $\xi\in\Hr_\et^1(X)$ to the locally constant map $x\mapsto x^*\xi$ from $X(\Rb)$ to $\Zb/2$. Now $x^*[\alpha]$ is classified by the double cover $\alpha_x:Y_x\to x=\Spec\Rb$ and is therefore zero if $Y_x=\Spec\Rb\coprod\Spec\Rb$ and $\omega\in\Hr_\et^1(\Rb)$ if $Y_x=\Spec\Cb$. We conclude that $\bar{\gamma}_1^0(\xi)$ is the tuple $(n_C)_{C\in\pi_0(X(\Rb))}$ where $n_C=0$ if there exists $z\in Z(\Rb)$ such that $p(z)$ lies in $C$ and $n_C=1$ else, which is precisely $[Z']$. This completes the proof.
\end{proof}

\begin{rema}
By using the localisation exact sequences in $\Hsc^*$-cohomology and mod $2$ cohomology of the real locus together with the four lemma, we can also deduce from Lemma \ref{lem:image_signature_I_surface} the image of $\gamma^0(X)$ if $X$ is the affine complement of a closed curve in a smooth surface $Y$ under the assumptions that the restriction map $\Ch^1(Y)\to\Ch^1(X)$ is injective and that $\overline{\gamma}_0^1(Y)$ is surjective. These assumptions hold, \emph{e.g.}, if $X$ is the complement of a curve in $\Abb_\Rb^2$, extending the work of Monnier \cite[proof of Theorem 4, p. 155]{monnierImageTotalSignature1997} on affine complements of hyperelliptic curves, or if $X$ is the complement of a curve of even degree in $\Pb_\Rb^2$.
\end{rema}

\section*{Declaration of generative AI and AI-assisted technologies in the manuscript preparation process}

The author used a conversational agent based on a large language model essentially as a literature-access tool to remedy his ignorance. Most prompts took the form ``Is it known that...'' or ``What are examples of...'', or were interpreted as requests of this type by the model when too imprecise. He investigated abelian varieties as a possible source of instances of the phenomenon exhibited in Example \ref{exe:connecting_homomorphism_nontrivial} following a prompt asking for examples of varieties with nontrivial $2$-torsion in the Picard group (though he later realised that this restriction is not necessary as it follows from Remark \ref{rema:connecting_homomorphism_nonzero} that this phenomenon also occurs for the real anisotropic quadric surface). He also benefited from interaction with the same model regarding the cohomology of abelian varieties, although in the final argument (which is, as far as he can tell, entirely the author's and is exposed in Example \ref{exe:connecting_homomorphism_nontrivial}), only the (very standard) description of the integral cohomology ring of a torus as an exterior algebra on the first cohomology group (\cite[Theorem 2.3]{milneAbelianVarieties2008}) is used. 
The author remains entirely responsible for the mathematical content of the article.

\printbibliography

\end{document}